%% file: man.tex
\documentclass[a4paper,11pt]{article}

\usepackage{amsmath,amssymb,amsthm,bm,mathtools,mathrsfs}
\usepackage{enumitem}
\usepackage[margin=30mm]{geometry}
\usepackage{graphicx,xcolor}
\usepackage{tikz}
\usetikzlibrary{arrows.meta}
\usepackage{float}
\usepackage{amsmath,amssymb,bm,amsthm}
\usepackage{ascmac}
\usepackage{array}
\usepackage[hang,small,bf]{caption}
\usepackage[subrefformat=parens]{subcaption}
\usepackage{natbib}
\usepackage{hyperref}
\hypersetup{
  colorlinks=true,
	citecolor=blue,
	linkcolor=red,
	urlcolor=orange,
}

\numberwithin{equation}{section}
\allowdisplaybreaks[3]

\newtheorem{theorem}{Theorem}[section]
\newtheorem{proposition}[theorem]{Proposition}
\newtheorem{lemma}[theorem]{Lemma}
\newtheorem{corollary}[theorem]{Corollary}

\theoremstyle{definition}
\newtheorem{definition}[theorem]{Definition}
\theoremstyle{remark}
\newtheorem{remark}[theorem]{Remark}
\newtheorem{example}[theorem]{Example}

\renewcommand{\theta}{\vartheta}

\title{Finite Cumulant Towers for Small-System-Noise Expansions in Nonlinear Filtering}
\author{Masahiro Kurisaki\thanks{Email: \href{mailto:masahiro.kurisaki@riken.jp}{\texttt{masahiro.kurisaki@riken.jp}}.}\\[0.4em]
  \small RIKEN Center for Advanced Intelligence Project}
\date{}

\begin{document}
\maketitle

\begin{abstract}
We study asymptotic expansions for nonlinear filtering in the small-system-noise regime. Our main result is that the expansion coefficients of conditional cumulants satisfy a closed finite-dimensional system that can be solved recursively. At each expansion degree, the evolution of a coefficient introduces a coefficient of the next higher cumulant rank, while a high-rank vanishing property terminates this hierarchy, yielding a finite triangular tower. The first layer is the Kalman--Bucy filter, and subsequent layers satisfy linear equations once the lower layers are known. To derive these equations, we establish generalized Zakai and Kushner--Stratonovich equations for observation-dependent random variables and obtain explicit coefficient equations in arbitrary dimension. The conditional cumulant expansion also determines the expansion of the conditional characteristic function. We further obtain an explicit bound on the number of scalar variables in terms of the state dimension and expansion degree. When the observation function vanishes, the construction reduces to ODEs for unconditional cumulant coefficients. Comparisons with existing asymptotic-expansion algorithms illustrate the resulting reduction in the number of ODE variables.
\end{abstract}

\section{Introduction}\label{section:introduction}

Filtering is the problem of estimating the conditional mean or distribution of an unobserved state $\{X_t\}_{t\geq 0}$ from a noisy observation $\{Y_t\}_{t\geq0}$. In the standard continuous-time setup, the state and observation processes $X$ and $Y$ are represented by the stochastic differential equations
\begin{align*}
  &dX_t=a(X_t)dt+b (X_t)dV_t,\\
  &dY_t=h(X_t)dt+\sigma(t)dW_t,
\end{align*}
where $V$ and $W$ are independent standard Brownian motions. In a linear Gaussian model, the conditional distribution $\pi_t$ of $X_t$ given $\mathcal{Y}_t=\sigma(Y_s;s\leq t)$ is Gaussian, and the evolution of its mean and covariance is determined by the finite-dimensional Kalman--Bucy equations \citep{kalman1961,bain_crisan2009filtering}.

For general nonlinear models, the evolution of the conditional distribution generally involves higher conditional moments. \citet{stratonovich1960,kushner1964} derived the Kushner--Stratonovich equation to describe the evolution of the conditional distribution $\pi_t$ under suitable conditions:
\begin{align*}
  d\pi_t(f)
  ={}&\pi_t(\mathcal L f)\,dt+\bigl\{\pi_t(fh^\top)-\pi_t(f)\pi_t(h)^\top\bigr\}
    (\sigma(t)\sigma(t)^\top)^{-1}\bigl\{dY_t-\pi_t(h)\,dt\bigr\},
\end{align*}
where
\begin{align*}
  &\pi_t(f)=E[f(X_t)|\mathcal{Y}_t],\\
  &\mathcal L f(x)
  =\sum_{i=1}^{d_1}a_i(x)\partial_i f(x)
    +\frac{1}{2}\sum_{i,j=1}^{d_1}
      (bb^\top)_{ij}(x)\partial_i\partial_j f(x).
\end{align*}
See also \citet{Liptser2001,bain_crisan2009filtering} for the proof of this equation. In particular, the evolution of the conditional expectation $E[X_t|\mathcal{Y}_t]$ requires conditional expectations involving $a(X_t)$ and $X_t h(X_t)^\top$, and applying the same equation to these quantities generally introduces an infinite number of variables. Thus, successive calculation of conditional moments generally leads to an infinite-dimensional problem \citep{HAZEWINKEL1983331}.

Under suitable regularity conditions, the Kushner--Stratonovich equation can be written as a stochastic partial differential equation for the conditional density, while the Zakai equation governs an unnormalized conditional density \citep{Zakai1969,bain_crisan2009filtering}. A straightforward computational approach is to apply numerical methods to these equations or their pathwise representations. Finite-difference \citep{kushner1976finite}, finite-element \citep{Germani01021988}, and spectral and Wiener-chaos methods \citep{lototsky1997spectral,lototsky2006wiener,lototsky2011chaos} have been developed for this purpose. However, standard tensor-product spatial discretizations with $N$ grid points or basis functions per coordinate require $N^{d_1}$ degrees of freedom. This exponential growth makes their application to high-dimensional filtering problems computationally difficult.

To avoid discretizing the entire state space, particle filters have become widely used alternatives \citep{gordon1993,4378823}. They approximate the conditional distribution by weighted particles and resampling, allowing general nonlinear and non-Gaussian models to be treated. Nevertheless, standard importance-sampling implementations suffer from weight degeneracy: in representative high-dimensional settings, preventing particle collapse requires a number of particles that grows exponentially with an effective dimension \citep{snyder2008obstacles}.

Computationally cheaper alternatives use linearization or Gaussian-based updates, as in the extended Kalman filter (EKF) \citep{picard1991} and the ensemble Kalman filter (EnKF) \citep{2003OcDyn..53..343E}. In particular, the EnKF combines nonlinear forecast dynamics with covariance-based linear updates and is widely used in high-dimensional geophysical applications \citep{evensenDataAssimilationEnsemble2009}. Quantifying the discrepancy from the true posterior distribution in general nonlinear models remains difficult. Rigorous accuracy results are available under additional assumptions, such as small noise for the EKF and full observations with small measurement noise for the ensemble Kalman--Bucy filter \citep{picard1991,dewiljesLongTimeStability2018}. Their accuracy can nevertheless deteriorate when strong nonlinearity produces substantial non-Gaussian features that are not captured by the linearization or the Kalman update \citep{2003OcDyn..53..343E}.

Against this background, small-noise expansions were proposed in \citet{kurisaki2026ks,kurisaki2026recursions}. In these works, we introduce a small parameter $\epsilon$ in the system noise:
\begin{align}
  &dX_t^\epsilon=a(X_t^\epsilon,0)dt+\epsilon\partial_\epsilon b(X_t^\epsilon,0)dV_t,\label{eq:intro-X}\\
  &dY_t^\epsilon=h(X_t^\epsilon,0)dt+\sigma(t)dW_t.\label{eq:intro-Y}
\end{align}
Then, we consider the expansion of the conditional expectation $E[X_t^\epsilon|\mathcal{Y}_t^\epsilon]$ with respect to $\epsilon$, where $\mathcal{Y}_t^\epsilon=\sigma(Y_s^\epsilon;s\leq t)$. In \citet{kurisaki2026ks}, we presented a rigorous justification of this expansion from the Kallianpur--Striebel formula, and a concrete formula to compute the expansion coefficients was given in \citet{kurisaki2026recursions}. The key observation is that the first-order approximation of \eqref{eq:intro-X} and \eqref{eq:intro-Y} is represented as a linear Gaussian model. Thus, we expand around this Gaussian reference, rather than using a direct Taylor expansion, so that the first-order approximation is the Kalman--Bucy filter.

The key advantage of the expansion approach is that the computation of the expansion coefficients reduces to a finite-dimensional system, and the number of required variables grows at most polynomially in the dimensions of $X$ and $Y$, rather than exponentially. In fact, \citet{kurisaki2026recursions} proved the finite-dimensional closure and provided a bound on the system dimension using Wick's formula and a generalized It\^o formula. Furthermore, higher-order corrections can improve accuracy beyond Gaussian-based approximations such as the EnKF. In the Lorenz--63 experiments of \citet{kurisaki2026lorenz}, the second-order method with recentering and redistribution achieved accuracy close to that of a particle filter under full observation, with lower truth RMSE than the ensemble transform Kalman filter (ETKF), although the method could become unstable when fewer state components were observed.

In the Wick-based construction, finite-dimensional closure follows relatively directly from the linear smoothing equations. However, deriving the evolution equations requires decomposing the moments in the expansion coefficients into products of means and covariances, differentiating the resulting terms, and then recombining them. This introduces many intermediate quantities and makes both the derivation and the resulting equations cumbersome, especially in multiple dimensions, which suggests that a more economical structure may underlie the individual Wick contractions.

In this work, motivated by this closure structure, we seek to identify that underlying structure and derive the evolution equations through a different route, without carrying out the Wick decomposition. Specifically, we apply the Kushner--Stratonovich equation successively to the expansion coefficients of the conditional expectation $E[X_t^\epsilon|\mathcal{Y}_t^\epsilon]$, rather than to the conditional expectation itself. However, the random variables inside the conditional expectations defining these coefficients depend on the observation path $Y$, so the usual equation does not apply directly. We therefore establish a generalized version allowing this observation dependence and prove that its successive application yields a finite closed system at each fixed expansion order. This yields explicit evolution equations for all quantities required to compute the expansion in the general multidimensional setting.

Moreover, the resulting equations not only provide another proof of the closure structure but also reveal a deeper structure related to the conditional distribution. The auxiliary variables introduced by the successive application of the Kushner--Stratonovich equation to the conditional expectation are themselves expansion coefficients of the conditional covariance and higher conditional cumulants. Thus, computing the conditional mean expansion simultaneously provides the coefficients of an expansion of the conditional distribution through its characteristic function. This also yields a sharper bound on the number of required variables than the one given in \citet{kurisaki2026recursions}.

Our expansion includes the unconditional expectation $E[f(X_t^\epsilon)]$ as a special case with $h=0$. Such expansions have been studied for decades in the context of option pricing in finance \citep{yoshida1992statistics,takahashi1999contingent,uchida2004option}. In particular, \citet{takahashi2012computation} compute characteristic-function expansion coefficients by deriving and solving ODEs for moments of products of diffusion expansion coefficients weighted by a Gaussian exponential. Our method organizes the required information directly through ODEs for cumulant coefficients, from which the characteristic-function expansion is recovered algebraically. This aggregation can substantially reduce the dimension of the ODE system compared with formulations that retain the individual weighted moments.

Section~\ref{section:overview} summarizes the tower structure. Section~\ref{section:preliminaries} recalls the model and the conditional expectation expansion. Section~\ref{section:cumulant-coefficients} introduces the cumulant coefficients and proves their high-rank vanishing property. Section~\ref{section:generalized-filtering-equations} establishes the generalized filtering equations, and Section~\ref{section:coefficient-evolution} derives the finite towers and the characteristic-function expansion. Section~\ref{section:unconditional} discusses unconditional expectations.

\section{Overview of the main results}\label{section:overview}
We briefly describe the structure of the expansion before giving its construction. Let us consider the small-noise system given by \eqref{eq:intro-X} and \eqref{eq:intro-Y}, and denote the coefficient of $\epsilon^k$ in the expansion of ${\rm cum}_r(X_t^\epsilon|\mathcal{Y}_t^\epsilon)$ by $A_{r,t}^{[k]}$. Then, we can write
\begin{align*}
  {\rm cum}_r(X_t^\epsilon|\mathcal{Y}_t^\epsilon)=A_{r,t}^{[0]}+A_{r,t}^{[1]}\epsilon+\cdots+A_{r,t}^{[n]}\epsilon^n+O_P^T(\epsilon^{n+1}).
\end{align*}
For the precise meaning of $O_P^T(\epsilon^{n+1})$, see Theorem \ref{theorem:filter-expansion}. In particular, the case of $r=1$ gives the expansion of the conditional expectation:
\begin{align*}
  E[X_t^\epsilon|\mathcal{Y}_t^\epsilon]=A_{1,t}^{[0]}+A_{1,t}^{[1]}\epsilon+\cdots+A_{1,t}^{[n]}\epsilon^n+O_P^T(\epsilon^{n+1}).
\end{align*}
Define the degree of the coefficient $A_{r,t}^{[k]}$ as
\begin{align*}
  {\rm deg}A_{r,t}^{[k]}=k-r+1.
\end{align*}
Then, our main result is stated as follows.
\begin{theorem}
  Under suitable assumptions, for $1\leq r<k$, the coefficient $A_{r,t}^{[k]}$ satisfies an equation of the form
  \begin{align*}
    dA_{r,t}^{[k]}
    ={}&\left\{\mathcal L_{r,t}\bigl[A_{r,t}^{[k]}\bigr]+F_{r,t}^{[k]}\right\}dt\\
    &+\left\{\mathcal H_{r,t}\bigl[A_{r+1,t}^{[k+1]}\bigr]+G_{r,t}^{[k]}\right\}
    \left\{dY_t^\epsilon-\left(h(X_t^{[0]},0)
    +\epsilon\partial_x h(X_t^{[0]},0)A_{1,t}^{[1]}\right)dt\right\}.
  \end{align*}
  Here, $\mathcal L_{r,t}$ and $\mathcal H_{r,t}$ are linear maps determined by the model and the Gaussian reference filter, $A_{r+1,t}^{[k+1]}$ is a new variable of degree ${\rm deg}A_{r,t}^{[k]}=k-r+1$, and the terms $F_{r,t}^{[k]}$ and $G_{r,t}^{[k]}$ are explicit polynomials in coefficients of degree less than $k-r+1$, with coefficients determined by the model. 
\end{theorem}
For the rigorous statement of this theorem, see Theorem \ref{theorem:coefficient-tower}. 

As a result, if coefficients of degree lower than $m$ are given, then the evolution of $A_{1,t}^{[m]}$ successively introduces new variables in the order
  \begin{align}
    A_{1,t}^{[m]},~A_{2,t}^{[m+1]},~A_{3,t}^{[m+2]},\cdots. \label{eq:degree-m-line}
  \end{align}
  Furthermore, another important result is that $A_{r,t}^{[k]}=0~~(k<2r-2)$ (see Theorem \ref{theorem:A-high-rank-vanishing}). This implies that the sequence of degree $m$ coefficients \eqref{eq:degree-m-line} actually terminates at $A_{m+1,t}^{[2m]}$, and all subsequent coefficients vanish. The resulting structure of the coefficient evolution is shown in Figure \ref{fig:cumulant-tower}. Note that the degree 0 variable $X_t^{[0]}$ is the solution of the deterministic equation
  \begin{align*}
    dX_t^{[0]}=a(X_t^{[0]},0)dt.
  \end{align*}

\begin{figure}[htbp]
  \centering
  \input{figures/cumulant_tower.tex}
  \caption{The first two diagonals of the cumulant tower, with the dependence on $t$ and $\epsilon$ in $A_{r,t}^{[k],\epsilon}$ suppressed. Columns align explicit powers of $\epsilon$. The arrows indicate the next coefficients introduced by the degree-one and degree-two evolution equations. The zeros are $A_3^{[3]}=0$ and $A_4^{[5]}=0$; the full conditional cumulants need not vanish.}
  \label{fig:cumulant-tower}
\end{figure}
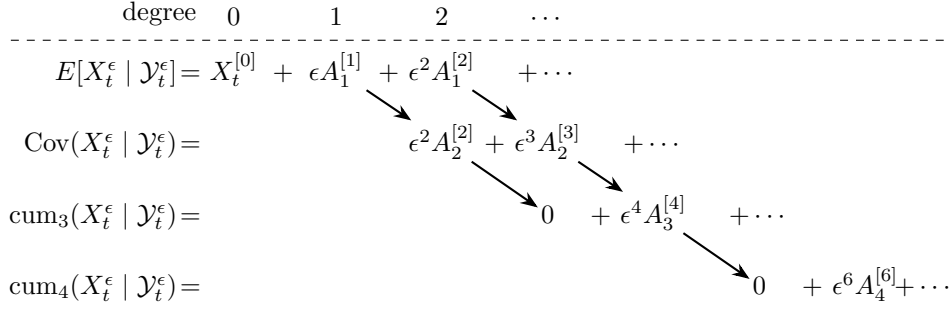

  For example, consider the one-dimensional model 
  \begin{align}
  &dX_t^\epsilon=a(X_t^\epsilon,0)dt+\epsilon dV_t,\label{eq:1dim-X}\\
  &dY_t^\epsilon=cX_t^\epsilon dt+\sigma dW_t,\label{eq:1dim-Y}
\end{align}
where $c\in\mathbb R$ and $\sigma>0$ are constants and $X_0=Y_0=0$. Write
\begin{align*}
  d\nu_t&=dY_t^\epsilon-c\bigl(X_t^{[0]}+\epsilon A_{1,t}^{[1]}\bigr)dt,\qquad
  \Gamma_t=\partial_x a(X_t^{[0]},0)-\frac{\epsilon^2c^2}{\sigma^2}A_{2,t}^{[2]}.
\end{align*}
Then, the degree-one equations are
\begin{align*}
  dA_{1,t}^{[1]}&=\partial_x a(X_t^{[0]},0)A_{1,t}^{[1]}dt
      +\frac{\epsilon c}{\sigma^2}A_{2,t}^{[2]}d\nu_t,\\
  dA_{2,t}^{[2]}&=\left\{2\partial_x a(X_t^{[0]},0)A_{2,t}^{[2]}+1
      -\frac{\epsilon^2c^2}{\sigma^2}\bigl(A_{2,t}^{[2]}\bigr)^2\right\}dt.
\end{align*}
Here, the evolution of $A_{1,t}^{[1]}$ introduces a new coefficient $A_{2,t}^{[2]}$ in degree $1$, and the evolution of the new variable introduces no new variable. These equations coincide exactly with the Kalman--Bucy equations for the following centered linearized model:
\begin{align*}
  &d\bar{X}_t=\partial_x a(X_t^{[0]},0)\bar{X}_t dt+dV_t,\\
  &d\bar{Y}_t=\epsilon c\bar{X}_t dt+\sigma dW_t.
\end{align*}  

For the second-order approximation, the additional degree-two coefficients satisfy
\begin{align*}
  dA_{1,t}^{[2]}&=\left\{\Gamma_t A_{1,t}^{[2]}+\frac{\partial_x^2 a(X_t^{[0]},0)}{2}
        \left(A_{2,t}^{[2]}+\bigl(A_{1,t}^{[1]}\bigr)^2\right)\right\}dt
      +\frac{\epsilon c}{\sigma^2}A_{2,t}^{[3]}d\nu_t,\\
  dA_{2,t}^{[3]}&=\left\{2\Gamma_t A_{2,t}^{[3]}
      +2\partial_x^2 a(X_t^{[0]},0)A_{1,t}^{[1]}A_{2,t}^{[2]}\right\}dt
      +\frac{\epsilon c}{\sigma^2}A_{3,t}^{[4]}d\nu_t,\\
  dA_{3,t}^{[4]}&=\left\{3\Gamma_t A_{3,t}^{[4]}
      +3\partial_x^2 a(X_t^{[0]},0)\bigl(A_{2,t}^{[2]}\bigr)^2\right\}dt.
\end{align*}
Since $A_{4,t}^{[5]}=0$, this tower closes after three variables. Once the degree-one coefficients $A_{1,t}^{[1]}$ and $A_{2,t}^{[2]}$ are known, these equations are solved in the order $A_{3,t}^{[4]}$, $A_{2,t}^{[3]}$, $A_{1,t}^{[2]}$.

  \section{Problem setup and preliminary results}\label{section:preliminaries}
  \subsection{Setup}
Let $(\Omega,\mathcal{F},\{\mathcal{F}_t\}_{t\geq 0},P)$ be a filtered probability space, and let $\{X_{t}^\epsilon\}_{t\geq 0}$ and $\{Y_t^\epsilon\}$ be $d_1$- and $d_2$-dimensional processes, respectively, satisfying
\begin{align}
  \label{eq-eta}&dX_t^\epsilon=a(X_t^\epsilon,\epsilon)dt+b(X_t^\epsilon,\epsilon)dV_t,~~X_0^\epsilon=x_0,\\
  \label{eq-Y-theta}&dY_t^\epsilon=h(X_t^\epsilon,\epsilon)dt+\sigma(t)dW_t,~~Y_0^\epsilon=0,
\end{align}
where $a:\mathbb{R}^{d_1}\times(-1,1)\to \mathbb{R}^{d_1}$, $b:\mathbb{R}^{d_1}\times(-1,1)\to M_{d_1,m_1}(\mathbb{R})$ and $h:\mathbb{R}^{d_1}\times(-1,1)\to \mathbb{R}^{d_2}$ are of class $C^\infty$ with
\begin{align*}
  b(x,0)=0
\end{align*}
for every $x$, and $x_0 \in \mathbb{R}^{d_1}$ is a constant. 
Here, $V$ and $W$ are independent $m_1$- and $m_2$-dimensional Brownian motions, respectively, with respect to $\{\mathcal F_t\}_{t\geq0}$, and $\sigma(t)\in M_{d_2,m_2}(\mathbb R)$ is deterministic and measurable.

Furthermore, we assume that for any $k \in \mathbb{Z}_+$
\begin{align*}
  \sup_{\substack{x\in\mathbb R^{d_1},|\epsilon|<1}}\left(|\partial_x^k\partial_\epsilon^q a(x,\epsilon)|+|\partial_x^k\partial_\epsilon^q b(x,\epsilon)|\right)<\infty,\qquad q\geq0
\end{align*}
and for every $k\geq0$ there exist constants $C_{k}$ and $q_{k}$ such that
\begin{align}
  \label{eq-assumption-h}\sup_{|\epsilon|<1}|\partial_x^j\partial_\epsilon^{k-j}h(x,\epsilon)|\leq C_{k}(1+|x|^{q_{k}}),\qquad 0\leq j\leq k.
\end{align}
Here, for $\phi \in C^k(\mathbb{R}^m; \mathbb{R}^n)$, the $k$-th derivative $\partial_x^k \phi(x)$ is regarded as an $\mathbb{R}^n$-valued $k$-tensor whose $(i_1,\ldots,i_k)$-th component is given by
  \begin{align*}
      \frac{\partial}{\partial x_{i_1}} \cdots \frac{\partial}{\partial x_{i_k}} \phi(x_1, \cdots, x_m).
  \end{align*}  
Furthermore, assume that for \( t \geq 0 \),
\begin{align}
  \int_0^t |\sigma(s)|^2 \, ds < \infty,
\end{align}
and that there exists a constant \( C > 0 \) such that for every \( t \geq 0 \),
\begin{align}
  \label{eq-assumption-sigma-2} \lambda_{\min}(\sigma(t) \sigma(t)^\top) > C,
\end{align}
where \( \lambda_{\min}(\sigma(t) \sigma(t)^\top) \) denotes the smallest eigenvalue of \( \sigma(t) \sigma(t)^\top \).
Here, $|A|$ denotes the Frobenius norm for a matrix $A$.

  Under this setup, we consider the asymptotic expansion of 
  \begin{align}
    \label{eq-conditional-expectation}E[f(X_t^\epsilon)|\mathcal{Y}_t^\epsilon]
  \end{align}
  where $\{\mathcal{Y}_t^\epsilon\}_{t\geq 0}$ is the filtration generated by $\{Y_t^\epsilon\}_{t\geq 0}$ augmented by null sets, and $f:\mathbb{R}^{d_1}\to \mathbb{R}$ is a polynomial.

  \subsection{Kallianpur--Striebel formula}
  We use the Kallianpur--Striebel formula to expand
  \eqref{eq-conditional-expectation}. It gives an explicit expression for the
  conditional expectation.

  Let us fix $T>0$, and define a probability measure $Q^\epsilon$ by
\begin{align*}
  Q^\epsilon(A) = E\left[ 1_A (Z_T^\epsilon)^{-1} \right]
\end{align*}
for \( A \in \mathcal{F} \), where
\begin{align*}
  \begin{split}
    Z_t^\epsilon &= \exp\Biggl( \int_0^t h(X_s^\epsilon,\epsilon)^\top (\sigma(s)\sigma(s)^\top)^{-1} \sigma(s) \, dW_s\\
     &\qquad+ \frac{1}{2} \int_0^t h(X_s^\epsilon,\epsilon)^\top (\sigma(s)\sigma(s)^\top)^{-1} h(X_s^\epsilon,\epsilon) \, ds \Biggr) \\    
    &= \exp\Biggl( \int_0^t h(X_s^\epsilon,\epsilon)^\top (\sigma(s)\sigma(s)^\top)^{-1} \, dY_s^\epsilon\\
     &\qquad- \frac{1}{2} \int_0^t h(X_s^\epsilon,\epsilon)^\top (\sigma(s)\sigma(s)^\top)^{-1} h(X_s^\epsilon,\epsilon) \, ds \Biggr).
  \end{split}
\end{align*}
Then, the conditional expectation (\ref{eq-conditional-expectation}) can be represented by the following Kallianpur--Striebel formula.
\begin{proposition}[Kallianpur--Striebel formula]
  For any $0\leq t\leq T$ and any polynomial $f:\mathbb{R}^{d_1}\to \mathbb{R}$, we have
  \begin{align}
  \label{eq-Kallianpur-Striebel-epsilon}
  E[f(X_t^\epsilon)|\mathcal{Y}_t^\epsilon] =
  \frac{\displaystyle E_{Q^\epsilon}\left[ f(X_t^\epsilon) Z_t^\epsilon \middle| \mathcal{Y}_t^\epsilon \right]}
       {\displaystyle E_{Q^\epsilon}\left[ Z_t^\epsilon \middle| \mathcal{Y}_t^\epsilon \right]},
\end{align}
where $E_{Q^\epsilon}$ is the expectation under $Q^\epsilon$.
\end{proposition}

\subsection{Expansion of the system}\label{section:expansion-of-system}
In order to consider the expansion of the quotient (\ref{eq-Kallianpur-Striebel-epsilon}), let us recall the differentiation of the state process $\{X_t^\epsilon\}_{t\geq 0}$ with respect to $\epsilon$. 

In what follows, we write
\begin{align*}
  &\partial_x^{j} a(x_1,\cdots,x_{d_1},0)[\xi^1\otimes \cdots \otimes \xi^j]
  =\sum_{i_1,\cdots,i_j=1}^{d_1}\frac{\partial}{\partial x_{i_1}} \cdots \frac{\partial}{\partial x_{i_j}} a(x_1, \cdots, x_{d_1},0)\xi_{i_1}^1\cdots \xi_{i_j}^j
\end{align*}
for $\xi^{i}=(\xi_1^i,\cdots,\xi_{d_1}^i)~(i=1,\cdots,j)$.

  \begin{proposition}\label{prop-expansion-of-SDE}
  For any fixed $t>0$ and $p \ge 1$, the mapping $\epsilon \to \{X_s^\epsilon\}_{0\leq s\leq t}$ is of class $C^\infty$ in the sense of $L^p(\Omega,C([0,t];\mathbb{R}^{d_1}))$, and the derivatives $X_t^{[k],\epsilon} := \partial_\epsilon^k X_t^\epsilon~(k\geq 1)$ satisfy the stochastic differential equation obtained by formal differentiation of \eqref{eq-eta}:
 \begin{align*}
  dX_t^{[k],\epsilon}
  ={}&\sum_{q=0}^k\binom{k}{q}\sum_{j=0}^{k-q}
    \sum_{(i_1,\ldots,i_j)\in\Lambda(k-q,j)}
    \nu(k-q,(i_1,\ldots,i_j))
    \partial_x^j\partial_\epsilon^q a(X_t^\epsilon,\epsilon)
    [X_t^{[i_1],\epsilon}\otimes\cdots\otimes X_t^{[i_j],\epsilon}]dt\\
  &+\sum_{q=0}^k\binom{k}{q}\sum_{j=0}^{k-q}
    \sum_{(i_1,\ldots,i_j)\in\Lambda(k-q,j)}
    \nu(k-q,(i_1,\ldots,i_j))
    \partial_x^j\partial_\epsilon^q b(X_t^\epsilon,\epsilon)
    [X_t^{[i_1],\epsilon}\otimes\cdots\otimes X_t^{[i_j],\epsilon}]dV_t,
\end{align*}
with $X_0^{[k],\epsilon}=0$. Here,
\begin{align*}
  \Lambda(k,j)=\left\{ (i_1,\cdots,i_j)\in \mathbb{N}^j;i_1\leq \cdots \leq i_j,~i_1+\cdots+i_j=k  \right\},
\end{align*}
and $\nu(k,(i_1,\cdots,i_j))$ represents the number of partitions of the set $\{1,\cdots,k\}$ into $j$ subsets with $(i_1,\cdots,i_j)$ elements. We adopt the conventions
\begin{align*}
  \Lambda(0,0)=\{()\},\qquad \nu(0,())=1,\qquad
  \Lambda(k,0)=\varnothing\quad(k\geq1),
\end{align*}
where $()$ denotes the empty tuple. Empty products are interpreted as $1$, and the zeroth derivative applied to an empty tensor product is interpreted as the function itself. 
\end{proposition}

Let us write $X_t^{[0]}=X_t^0$ and $X_t^{[k]}=X_t^{[k],0}$ for $k\geq1$. Then, $X^{[0]}$ is the solution of the deterministic equation 
\begin{align*}
  \frac{dX_t^{[0]}}{dt}=a(X_t^{[0]},0),~~X_0^{[0]}=x_0,
\end{align*}
and $X^{[1]}$ is a Gaussian process given by
\begin{align}
  \label{eq:X1}dX_t^{[1]}=\left\{\partial_x a(X_t^{[0]},0)X_t^{[1]}+\partial_\epsilon a(X_t^{[0]},0)\right\}dt+\partial_\epsilon b(X_t^{[0]},0)dV_t.
\end{align}
In general, $X_t^{[p]}~(p\geq 2)$ is the solution of
\begin{align}
  \begin{split}\label{eq-Xp}
    dX_t^{[p]}=&\partial_x a(X_t^{[0]},0)[X_t^{[p]}]dt\\
  &+\sum_{j=2}^p \sum_{(i_1,\cdots,i_j)\in \Lambda(p,j)}\nu(p,(i_1,\cdots,i_j))\partial_x^{j} a(X_t^{[0]},0)[X_t^{[i_1]}\otimes \cdots \otimes X_t^{[i_j]}]dt\\
  &+\sum_{q=1}^p\binom pq\sum_{j=0}^{p-q}\sum_{(i_1,\ldots,i_j)\in\Lambda(p-q,j)}
    \nu(p-q,(i_1,\ldots,i_j))
    \partial_x^j\partial_\epsilon^q a(X_t^{[0]},0)
    [X_t^{[i_1]}\otimes\cdots\otimes X_t^{[i_j]}]dt\\
  &+\sum_{q=2}^p\binom pq\sum_{j=0}^{p-q}\sum_{(i_1,\ldots,i_j)\in\Lambda(p-q,j)}
    \nu(p-q,(i_1,\ldots,i_j))
    \partial_x^j\partial_\epsilon^q b(X_t^{[0]},0)
    [X_t^{[i_1]}\otimes\cdots\otimes X_t^{[i_j]}]dV_t\\
  &+p\sum_{j=0}^{p-1} \sum_{(i_1,\cdots,i_j)\in \Lambda(p-1,j)}\nu(p-1,(i_1,\cdots,i_j))\partial_x^{j}\partial_\epsilon b(X_t^{[0]},0)[X_t^{[i_1]}\otimes \cdots \otimes X_t^{[i_j]}]dV_t,
  \end{split}  
\end{align}
and thus it can be written as 
\begin{align}
  \label{eq:Xp}
  \begin{split}
    &X_t^{[p]}\\
    =&\sum_{j=2}^p \sum_{(i_1,\cdots,i_j)\in \Lambda(p,j)}\nu(p,(i_1,\cdots,i_j))\int_0^t \Phi(s,t)\partial_x^{j} a(X_s^{[0]},0)[X_s^{[i_1]}\otimes \cdots \otimes X_s^{[i_j]}]ds\\
  &+\sum_{q=1}^p\binom pq\sum_{j=0}^{p-q}\sum_{(i_1,\ldots,i_j)\in\Lambda(p-q,j)}
    \nu(p-q,(i_1,\ldots,i_j))
    \int_0^t\Phi(s,t)\partial_x^j\partial_\epsilon^q a(X_s^{[0]},0)
    [X_s^{[i_1]}\otimes\cdots\otimes X_s^{[i_j]}]ds\\
  &+\sum_{q=2}^p\binom pq\sum_{j=0}^{p-q}\sum_{(i_1,\ldots,i_j)\in\Lambda(p-q,j)}
    \nu(p-q,(i_1,\ldots,i_j))
    \int_0^t\Phi(s,t)\partial_x^j\partial_\epsilon^q b(X_s^{[0]},0)
    [X_s^{[i_1]}\otimes\cdots\otimes X_s^{[i_j]}]dV_s\\
  &+p\sum_{j=0}^{p-1} \sum_{(i_1,\cdots,i_j)\in \Lambda(p-1,j)}\nu(p-1,(i_1,\cdots,i_j))\int_0^t \Phi(s,t)\partial_x^{j}\partial_\epsilon b(X_s^{[0]},0)[X_s^{[i_1]}\otimes \cdots \otimes X_s^{[i_j]}]dV_s,
  \end{split}  
\end{align}
where $\Phi(s,t)$ is the solution of the $d_1\times d_1$-matrix equation 
\begin{align*}
  \frac{\partial}{\partial t}\Phi(s,t)
  =\partial_x a(X_t^{[0]},0)\Phi(s,t),~~\Phi(s,s)=I_{d_1}.
\end{align*}
In what follows, we write
\begin{align*}
  \partial_x^j\partial_\epsilon^q \phi(X_t^{[0]},0)= \phi_t^{[j,q]},\qquad \partial^k f(X_t^{[0]})=f_t^{[k]}\quad(k\geq0)
\end{align*}
for $\phi=a,b,h,\beta$, where $\beta=b b^\top$.

\subsection{Expansion of the conditional expectation}
To describe the expansion of the conditional expectation (\ref{eq-conditional-expectation}), let us write 
\begin{align*}
  F_t^{[k]}:=\left.\frac{1}{k!}\frac{\partial^k}{\partial \epsilon^k}f(X_t^\epsilon)\right|_{\epsilon=0},~~H_t^{[k]}:=\left.\frac{1}{k!}\frac{\partial^k}{\partial \epsilon^k}h(X_t^\epsilon,\epsilon)\right|_{\epsilon=0}~~(k\geq 0).
\end{align*}
Note that $f_t^{[0]}$ and $h_t^{[0,0]}$ are deterministic processes, and for $k\geq 1$, we have the expression
\begin{align}
  \label{eq:Fk}&F_t^{[k]}=\frac{1}{k!}\sum_{j=1}^k\sum_{(i_1,\cdots,i_j)\in \Lambda(k,j)}\nu(k,(i_1,\cdots,i_j))f_t^{[j]}[X_t^{[i_1]}\otimes \cdots \otimes X_t^{[i_j]}],\\
  \begin{split}\label{eq:Hk}
  &H_t^{[k]}={}\frac1{k!}\sum_{q=0}^k\binom{k}{q}\sum_{j=0}^{k-q}
    \sum_{(i_1,\ldots,i_j)\in\Lambda(k-q,j)}\nu(k-q,(i_1,\ldots,i_j))
    h_t^{[j,q]}
    [X_t^{[i_1]}\otimes\cdots\otimes X_t^{[i_j]}].
  \end{split}
\end{align}
Fix $\lambda\in[0,1)$. For $0<\epsilon<1$, define
\begin{align*}
  &\mathsf H_t^{\lambda,\epsilon}=h_t^{[0,0]}+\frac{\lambda}{\epsilon}\{h(X_t^{\epsilon},\epsilon)-h_t^{[0,0]}\},\\
  &d\mathsf{Y}_t^{\lambda,\epsilon}=\mathsf H_t^{\lambda,\epsilon}dt+\sigma(t)dW_t,\qquad \mathsf Y_0^{\lambda,\epsilon}=0.
\end{align*}
At $\epsilon=0$, set $\mathsf H_t^{\lambda,0}=H_t^{[0]}+\lambda H_t^{[1]}$ and define $\mathsf Y^{\lambda,0}$ by the same observation equation. In particular, we have
\begin{align*}
  \mathsf H_t^{\epsilon,\epsilon}=h(X_t^\epsilon,\epsilon),~~\mathsf{Y}_t^{\epsilon,\epsilon}=Y_t^\epsilon.
\end{align*}
Moreover, we have the expansion
\begin{align*}
  \mathsf H_t^{\lambda,\epsilon}=H_t^{[0]}+\lambda \left( H_t^{[1]}+\sum_{i=2}^n H_t^{[i]}\epsilon^{i-1} \right)+r_t^{\lambda,\epsilon,n+1},
\end{align*}
where
\begin{align*}
  r_t^{\lambda,\epsilon,n+1}=\frac{\lambda\epsilon^n}{n!}\int_0^1(1-u)^n
  \left.\frac{\partial^{n+1}}{\partial\theta^{n+1}}h(X_t^\theta,\theta)\right|_{\theta=\epsilon u}du.
\end{align*}

Let $R(t)=\sigma(t)\sigma(t)^\top$, and let $\{\mathscr Y_t^{\lambda,\epsilon}\}_{t\geq0}$ be the filtration generated by $\mathsf Y^{\lambda,\epsilon}$ augmented with null sets. Define
\begin{align}
  &\mathsf Z_t^{\lambda,\epsilon}
  =\exp\left(\int_0^t(\mathsf H_s^{\lambda,\epsilon})^\top R(s)^{-1}\,d\mathsf Y_s^{\lambda,\epsilon}
  -\frac12\int_0^t(\mathsf H_s^{\lambda,\epsilon})^\top R(s)^{-1}\mathsf H_s^{\lambda,\epsilon}\,ds\right),\nonumber\\
  &\mathsf Q^{\lambda,\epsilon}(A)=E\left[1_A(\mathsf Z_T^{\lambda,\epsilon})^{-1}\right],\qquad A\in\mathcal F,\nonumber\\
  &\tilde H_t^\lambda=H_t^{[0]}+\lambda H_t^{[1]},\label{eq:def-H-tilde}
\end{align}
and write
\begin{align*}
  \tilde{Z}_t^{\lambda,\epsilon}=&\exp\left( \int_0^t (\tilde{H}_s^\lambda)^\top R(s)^{-1} \, d\mathsf{Y}_s^{\lambda,\epsilon}  - \frac{1}{2} \int_0^t (\tilde{H}_s^\lambda)^\top R(s)^{-1} \tilde{H}_s^\lambda \, ds \right).
\end{align*}
For tensor-valued random variables $U$ with $E_{\mathsf Q^{\lambda,\epsilon}}[|U|\tilde{Z}_t^{\lambda,\epsilon}]<\infty$, define
\begin{align}
  \tilde{\mathsf E}_t^{\lambda,\epsilon}[U]=\frac{E_{\mathsf Q^{\lambda,\epsilon}}[U\tilde{Z}_t^{\lambda,\epsilon}|\mathscr{Y}_t^{\lambda,\epsilon}]}{E_{\mathsf Q^{\lambda,\epsilon}}[\tilde{Z}_t^{\lambda,\epsilon}|\mathscr{Y}_t^{\lambda,\epsilon}]}.\label{eq:def-E-tilde}
\end{align}
Then, we have the following result.
\begin{proposition}
  For $0\leq t\leq T$, let 
  \begin{align*}
    \Phi:C([0,t];\mathbb{R}^{m_1})\times C([0,t];\mathbb{R}^{d_2})\to \mathbb{R}^k
  \end{align*}
  be a measurable map, and define
  \begin{align*}
    \bar{\mathsf Y}_t^{\lambda,\epsilon}=\int_0^t \tilde{H}_s^\lambda\,ds+\int_0^t \sigma(s)dW_s.
  \end{align*}  
  Then, for 
  \begin{align*}
    &U=\Phi(\{V_s\}_{0\leq s\leq t},\{\mathsf Y_s^{\lambda,\epsilon}\}_{0\leq s\leq t}),\\
    &\bar{U}=\Phi(\{V_s\}_{0\leq s\leq t},\{\bar{\mathsf Y}_s^{\lambda,\epsilon}\}_{0\leq s\leq t})
  \end{align*}
  we have
  \begin{align*}
   E_{\mathsf Q^{\lambda,\epsilon}}[|U|\tilde{Z}_t^{\lambda,\epsilon}]=E[|\bar{U}|].
  \end{align*}
  In particular, $E_{\mathsf Q^{\lambda,\epsilon}}[|U|\tilde{Z}_t^{\lambda,\epsilon}]<\infty$ if $E[|\bar{U}|]<\infty$, and thus $\tilde{\mathsf E}_t^{\lambda,\epsilon}[U]$ is well-defined.
\end{proposition}
This result easily follows from Girsanov's theorem. See the argument following Proposition 4.5 of \citet{kurisaki2026ks} for details. In the sequel, we often use this result without explicit mention.

For later reference, we define the functions
  \begin{align}
    &\mathbb{H}_t^{(n),\lambda}(z)=\lambda \sum_{i=2}^{n+1} H_t^{[i]}z^{i-1} ,\label{eq:def-H-bb}\\
    \begin{split}
      &\mathbb{J}_t^{(n),\lambda,\epsilon}(z)=1+\sum_{r=1}^n\int_0^t\int_0^{t_r}\cdots\int_0^{t_2}
      \mathbb{H}_{t_1}^{(n),\lambda}(z)^\top R({t_1})^{-1}
      \{d\mathsf Y_{t_1}^{\lambda,\epsilon}-\tilde{H}_{t_1}^\lambda\,dt_1\}\\
    &\qquad\qquad\times\cdots\times
      \mathbb{H}_{t_r}^{(n),\lambda}(z)^\top R({t_r})^{-1}
      \{d\mathsf Y_{t_r}^{\lambda,\epsilon}-\tilde{H}_{t_r}^\lambda\,dt_r\}      
    \end{split}
    \label{eq:def-J-bb},
  \end{align}
  for $z \in \mathbb{R}$.

Then, Theorem 4.10 of \citet{kurisaki2026ks} yields the following result.
\begin{theorem}\label{theorem:filter-expansion}
  We have
  \begin{align}
    \begin{split}
      &E[f(X_t^\epsilon)\mid\mathscr Y_t^{\lambda,\epsilon}]\\
    =&\Biggl(\sum_{k=0}^{n}\tilde {\mathsf E}_t^{\lambda,\epsilon}
      [F_t^{[k]}\mathbb{J}_t^{(n),\lambda,\epsilon}(\epsilon)]\epsilon^k\Biggr)\Biggl\{1+\sum_{j=1}^{n}(-1)^j
      \Biggl(\tilde {\mathsf E}_t^{\lambda,\epsilon}
      [\mathbb{J}_t^{(n),\lambda,\epsilon}(\epsilon)]-1\Biggr)^j\Biggr\}
      +O_P^T(\epsilon^{n+1}),
    \end{split}\label{eq:expansion}    
  \end{align}
  for every $n\in\mathbb N$ and $0\leq t\leq T$, as $\epsilon\to0$ with $\lambda$ fixed. Here, $\eta^{\lambda,\epsilon}=O_P^T(\epsilon^k)$ means that, for some $\epsilon_0\in(0,1)$,
  \begin{align*}
    \lim_{K\to\infty}\sup_{0\leq t\leq T}\sup_{0<\epsilon<\epsilon_0}
    P(|\eta_t^{\lambda,\epsilon}|>K\epsilon^k)=0.
  \end{align*}
  The remainder estimate is stated for each fixed $\lambda$.
\end{theorem}
\begin{proof}
  Apply Theorem 4.10 of \citet{kurisaki2026ks} with perturbation parameter $\epsilon$, reference observation drift $\tilde H^{\lambda}$, observation coefficients $\lambda H^{[r+1]}$ at order $r\geq1$, and $U_t^\epsilon=f(X_t^\epsilon)$.
\end{proof}
\begin{remark}
  In what follows, we may set $\lambda=\epsilon$ when applying this theorem or Theorem 4.10 of \citet{kurisaki2026ks}, without changing the stated remainder orders. Strictly speaking, this falls outside the original formulation since the reference observation drift $\tilde H^{\epsilon}$ and the observation coefficients $\epsilon H^{[r+1]}$ depend on $\epsilon$. 
  
  However, the bounds required in Assumptions {[A1]}--{[A4]} of the reference hold uniformly for $\epsilon\in[0,1)$ in our setting. The change-of-measure and inverse-denominator estimates used in the proof also remain uniform, so the same argument yields the stated remainder estimates.
\end{remark}

Collecting the explicit powers of $\epsilon$ in \eqref{eq:expansion}, define, for $k\geq0$,
\begin{align}
  \label{eq:moment-expansion-coefficients}
  \begin{split}
    \mathsf M_t^{[k],\lambda,\epsilon}(f)
    &:=\sum_{j=0}^k(-1)^j
      \sum_{\substack{p,q_0\geq0,\ q_1,\ldots,q_j\geq1\\
        p+q_0+\cdots+q_j=k}}
      \tilde{\mathsf E}_t^{\lambda,\epsilon}
        [F_t^{[p]}\mathsf I_t^{[q_0],\lambda,\epsilon}]\prod_{\ell=1}^j
      \tilde{\mathsf E}_t^{\lambda,\epsilon}
        [\mathsf I_t^{[q_\ell],\lambda,\epsilon}],
  \end{split}
\end{align}
where for $r\geq1$ and $k_1,\ldots,k_r\geq2$,
  \begin{align}
    \label{def-J-i}
    \begin{split}
    \mathsf J_t^{(k_1,\ldots,k_r),\lambda,\epsilon}&=\int_0^t\int_0^{t_r}\cdots\int_0^{t_2}
      (H_{t_1}^{[k_1]})^\top R({t_1})^{-1}
      \{d\mathsf Y_{t_1}^{\lambda,\epsilon}-\tilde{H}_{t_1}^\lambda\,dt_1\}\\
    &\qquad\qquad\times\cdots\times
      (H_{t_r}^{[k_r]})^\top R({t_r})^{-1}
      \{d\mathsf Y_{t_r}^{\lambda,\epsilon}-\tilde{H}_{t_r}^\lambda\,dt_r\}.
    \end{split}
  \end{align}
  and
  \begin{align}
    \label{def-I-coefficients}
    \mathsf I_t^{[0],\lambda,\epsilon}=1,\quad 
    \mathsf I_t^{[k],\lambda,\epsilon}
    =\sum_{r=1}^k\lambda^r
      \sum_{\substack{k_1,\ldots,k_r\geq2\\k_1+\cdots+k_r=k+r}}
      \mathsf J_t^{(k_1,\ldots,k_r),\lambda,\epsilon}~~(k\geq 1).
  \end{align}

For $j=0$, the indices $q_1,\ldots,q_j$ are absent and the product is understood as $1$. Then \eqref{eq:expansion} can be written as
\begin{align}
  \label{eq:collected-moment-expansion}
  E[f(X_t^\epsilon)\mid\mathscr Y_t^{\lambda,\epsilon}]
  =\sum_{k=0}^n\mathsf M_t^{[k],\lambda,\epsilon}(f)\epsilon^k
    +O_P^T(\epsilon^{n+1}),
\end{align}
as $\epsilon\to0$ with $\lambda$ fixed.

\section{Cumulant coefficients and generating function}
\label{section:cumulant-coefficients}
\subsection{Conditional cumulant}\label{section:conditional-cumulant}
For $r\geq1$, let $\mathcal P_r$ denote the set of partitions of $\{1,\ldots,r\}$. Define the tensor ${\rm cum}_r(X_t^\epsilon\mid\mathscr Y_t^{\lambda,\epsilon})$ componentwise, for $i_1,\ldots,i_r\in\{1,\ldots,d_1\}$, by
\begin{align}
  \bigl({\rm cum}_{r}(X_t^\epsilon\mid\mathscr{Y}_t^{\lambda,\epsilon})\bigr)_{i_1\cdots i_r}
  =\sum_{\pi\in\mathcal P_r}(-1)^{|\pi|-1}(|\pi|-1)!
    \prod_{B\in\pi}E\!\left[\prod_{a\in B}X_{t,i_a}^\epsilon
      \,\middle|\,\mathscr Y_t^{\lambda,\epsilon}\right],\label{eq:conditional-cumulant}
\end{align}
where $X_{t,i}^\epsilon$ is the $i$-th element of $X_t^\epsilon$. This is the conditional $r$-th cumulant of $X_t^\epsilon$. 

Consider its asymptotic expansion with respect to $\epsilon$ by applying Theorem \ref{theorem:filter-expansion} to each moment in the expression. For each $B\subset \{1,\cdots,r\}$, set $f_B(x)=\prod_{a\in B}x_{i_a}$. Since \eqref{eq:collected-moment-expansion} gives
\begin{align*}
  E\!\left[\prod_{a\in B}X_{t,i_a}^\epsilon \,\middle|\,\mathscr Y_t^{\lambda,\epsilon}\right]
  =\sum_{q=0}^n\mathsf M_t^{[q],\lambda,\epsilon}(f_B)\epsilon^q
    +O_P^T(\epsilon^{n+1}),
\end{align*}
we define $\bigl(\mathsf A_{r,t}^{[k],\lambda,\epsilon}\bigr)_{i_1\cdots i_r}$, for $0\leq k\leq n$, as the coefficient of $z^k$ in the polynomial
\begin{align}
  \label{eq:cumulant-coefficient-polynomial}
  \sum_{\pi\in\mathcal P_r}(-1)^{|\pi|-1}(|\pi|-1)!
    \prod_{\ell=1}^{|\pi|}
    \left(\sum_{q=0}^n\mathsf M_t^{[q],\lambda,\epsilon}(f_{\pi_\ell})z^q\right).
\end{align}
Here, $z$ is a formal variable, $|\pi|$ denotes the number of blocks of $\pi$, and $\pi_1,\ldots,\pi_{|\pi|}$ is any enumeration of these blocks. The expression is independent of the choice of enumeration, and its coefficient of $z^k$ is independent of $n$ whenever $n\geq k$. Then we have
\begin{align*}
  {\rm cum}_{r}(X_t^\epsilon|\mathscr{Y}_t^{\lambda,\epsilon})=\sum_{k=0}^n \mathsf A_{r,t}^{[k],\lambda,\epsilon}\epsilon^k+O_P^T(\epsilon^{n+1})
\end{align*}
(see the remark following Definition 4.6 in \citet{kurisaki2026ks}). In particular, if we write
\begin{align*}
  &A_{r,t}^{[k],\epsilon}=\mathsf A_{r,t}^{[k],\epsilon,\epsilon},\qquad I_t^{[k],\epsilon}:=\mathsf I_t^{[k],\epsilon,\epsilon},\qquad J_t^{(k_1,\ldots,k_r),\epsilon}:=\mathsf J_t^{(k_1,\ldots,k_r),\epsilon,\epsilon},\qquad\tilde E_t^\epsilon=\tilde{\mathsf E}_t^{\epsilon,\epsilon},
\end{align*}
then we can write
\begin{align}
  {\rm cum}_r(X_t^{\epsilon}|\mathcal{Y}_t^\epsilon)=\sum_{k=0}^n A_{r,t}^{[k],\epsilon}\epsilon^k+O_P^T(\epsilon^{n+1}).\label{eq:expansion-cumulant}
\end{align}

In the case of $r=1$, this is the expansion of the conditional expectation $E[X_t^\epsilon|\mathcal{Y}_t^\epsilon]$, and the coefficients up to the third order are
\begin{align*}
  A_{1,t}^{[0],\epsilon}&=X_t^{[0]},\qquad
  A_{1,t}^{[1],\epsilon}=\tilde E_t^\epsilon[X_t^{[1]}],\\
  A_{1,t}^{[2],\epsilon}
  &=\frac12\tilde E_t^\epsilon[X_t^{[2]}]
    +\epsilon\Bigl\{\tilde E_t^\epsilon[X_t^{[1]}J_t^{(2),\epsilon}]
    -\tilde E_t^\epsilon[X_t^{[1]}]\tilde E_t^\epsilon[J_t^{(2),\epsilon}]\Bigr\},\\
  A_{1,t}^{[3],\epsilon}
  &=\frac16\tilde E_t^\epsilon[X_t^{[3]}]+\frac\epsilon2\Bigl\{\tilde E_t^\epsilon[X_t^{[2]}J_t^{(2),\epsilon}]
    -\tilde E_t^\epsilon[X_t^{[2]}]\tilde E_t^\epsilon[J_t^{(2),\epsilon}]\Bigr\}\\
  &\quad+\epsilon\Bigl\{\tilde E_t^\epsilon[X_t^{[1]}J_t^{(3),\epsilon}]
    -\tilde E_t^\epsilon[X_t^{[1]}]\tilde E_t^\epsilon[J_t^{(3),\epsilon}]\Bigr\}\\
  &\quad+\epsilon^2\Bigl\{\tilde E_t^\epsilon[X_t^{[1]}J_t^{(2,2),\epsilon}]
    -\tilde E_t^\epsilon[X_t^{[1]}]\tilde E_t^\epsilon[J_t^{(2,2),\epsilon}]\Bigr\}\\
  &\quad-\epsilon^2\tilde E_t^\epsilon[J_t^{(2),\epsilon}]
    \Bigl\{\tilde E_t^\epsilon[X_t^{[1]}J_t^{(2),\epsilon}]
    -\tilde E_t^\epsilon[X_t^{[1]}]\tilde E_t^\epsilon[J_t^{(2),\epsilon}]\Bigr\}.
\end{align*}
In general, the coefficients $A_{1,t}^{[k],\epsilon}$ contain additional explicit powers of $\epsilon$ introduced by setting $\lambda=\epsilon$. They therefore differ from the coefficients obtained by collecting explicit powers of $\epsilon$ directly in the original expansion of $E[X_t^\epsilon\mid\mathcal Y_t^\epsilon]$. For example, the latter coefficients at orders two and three are
\begin{align*}
  &\tilde{A}_{1,t}^{[2],\epsilon}=\frac12\tilde E_t^\epsilon[X_t^{[2]}],\\
  &\tilde{A}_{1,t}^{[3],\epsilon}=\frac16\tilde E_t^\epsilon[X_t^{[3]}]+\tilde E_t^\epsilon[X_t^{[1]}J_t^{(2),\epsilon}]
    -\tilde E_t^\epsilon[X_t^{[1]}]\tilde E_t^\epsilon[J_t^{(2),\epsilon}].
\end{align*} 
Although $\tilde A_{1,t}^{[k],\epsilon}$ have simpler expressions, we use $A_{1,t}^{[k],\epsilon}$ because their time evolution can be described using fewer auxiliary variables.

\subsection{Generating function}\label{sec:generating-function}
Throughout this subsection, we fix $t,\lambda,\epsilon$ and use versions of conditional expectations induced by a common regular conditional distribution under $\mathsf Q^{\lambda,\epsilon}$ of the random variables appearing below, given $\mathscr Y_t^{\lambda,\epsilon}$.
Consider the function
\begin{align*}
  \Phi_t^{\lambda,\epsilon}(u)= E_{\mathsf Q^{\lambda,\epsilon}}\left[ e^{iu\cdot X_t^\epsilon}\mathsf{Z}_t^{\lambda,\epsilon}\middle|\mathscr Y_t^{\lambda,\epsilon} \right]~~(u \in \mathbb{R}^{d_1}).
\end{align*}
Taking the logarithm locally near $u=0$, we have
\begin{align*}
  \frac{1}{i}\partial_u\log \Phi_t^{\lambda,\epsilon}(u)|_{u=0}=\frac{E_{\mathsf Q^{\lambda,\epsilon}}[ X_t^\epsilon\mathsf{Z}_t^{\lambda,\epsilon}|\mathscr Y_t^{\lambda,\epsilon} ]}{E_{\mathsf Q^{\lambda,\epsilon}}[ \mathsf{Z}_t^{\lambda,\epsilon}|\mathscr Y_t^{\lambda,\epsilon} ]}=E[X_t^\epsilon|\mathscr Y_t^{\lambda,\epsilon}].
\end{align*}
More generally, we have
\begin{align*}
  \frac{1}{i^r}\partial_u^r\log \Phi_t^{\lambda,\epsilon}(u)|_{u=0}={\rm cum}_r(X_t^\epsilon\mid\mathscr Y_t^{\lambda,\epsilon}),
\end{align*}
and $\mathsf A_{r,t}^{[k],\lambda,\epsilon}$ is its coefficient of $\epsilon^k$ in the expansion.

We now construct a generating function for $\mathsf A_{r,t}^{[k],\lambda,\epsilon}$. Fix $n\geq1$ and let $z\in\mathbb R$. Define for $u\in\mathbb R^{d_1}$
\begin{align}
  &\mathbb{X}_t^{(n)}(z)=\sum_{r=1}^n\frac{z^r}{r!}X_t^{[r]},\label{eq:def-bb-X} \\
  \label{eq:def-generating-K}
  \begin{split}
    &K_t^{\lambda,\epsilon,(n)}(u,z)=\sum_{p=1}^n\frac{(-1)^{p-1}}p
    \left\{\tilde{\mathsf E}_t^{\lambda,\epsilon}\left[
    \sum_{\ell=0}^n\frac1{\ell!}
    \left(iu\cdot\mathbb{X}_t^{(n)}(z)\right)^\ell
    \mathbb{J}_t^{(n),\lambda,\epsilon}(z)
    \right]-1\right\}^p.
  \end{split}
\end{align}
Coefficients in $z$ below are Taylor coefficients at $z=0$, with $(\lambda,\epsilon)$ fixed. In the sequel, we fix \(\lambda,\epsilon\) and \(n\) unless otherwise stated, and suppress the dependence on these parameters in the notation.

Then the following result holds.

\begin{proposition}\label{prop:A-derivative-K}
  For any integer $r\geq1$ and $1\leq k\leq n$, the coefficient of $z^k$ in the polynomial
  \begin{align*}
    \frac1{i^r}\partial_u^r K_t(u,z)|_{u=0}
  \end{align*}
  is equal to $\mathsf A_{r,t}^{[k]}$.
\end{proposition}
\begin{proof}
  Write
  \begin{align*}
    &C_t(z)=1+\sum_{\ell=1}^n(-1)^\ell
    \left(\tilde{\mathsf E}_t
      [\mathbb J_t(z)]-1\right)^\ell,\\
    &G_t(u,z)=\tilde{\mathsf E}_t\left[ \sum_{\ell=0}^n\frac1{\ell!}\left(iu\cdot \mathbb X_t(z) \right)^\ell
  \mathbb{J}_t(z)
   \right]
  \end{align*}
  and consider
  \begin{align}
    \tilde{K}_t(u,z)=&\sum_{p=1}^n\frac{(-1)^{p-1}}p
  \left\{C_t(z)G_t(u,z)-1\right\}^p.\label{eq:def-Ktilde}
  \end{align}
  Then, it is straightforward to check that for $1\leq k\leq n$, the coefficient of $z^k$ in 
  \begin{align*}
    \partial_u K_t(u,z)-\partial_u \tilde{K}_t(u,z)
  \end{align*}
  equals zero. Further differentiation in $u$ preserves this property. Thus we can consider $\tilde{K}_t$ instead of $K_t$.

  Consider the conditional characteristic function
\begin{align*}
  \phi(u)=E[e^{iu\cdot X_t^\epsilon}|\mathscr{Y}_t^{\lambda,\epsilon}].
\end{align*}
The conditional moments of all orders are finite, so $\phi$ admits a smooth version in $u$. Taking the branch of the logarithm that vanishes at $u=0$, we have the expansion
\begin{align*}
  \log\phi(u)=\sum_{p=1}^\infty \frac{(-1)^{p-1}}{p}(\phi(u)-1)^p
\end{align*}
on a possibly random neighborhood of $u=0$. Let us define
\begin{align*}
  \psi^{(r)}(u)=\sum_{p=1}^r \frac{(-1)^{p-1}}{p}(\phi(u)-1)^p.
\end{align*}
Then, since $\phi(0)=1$, the derivatives of $\log\phi-\psi^{(r)}$ up to order $r$ vanish at $u=0$. Hence
\begin{align*}
  i^r{\rm cum}_r(X_t^\epsilon\mid\mathscr{Y}_t^{\lambda,\epsilon})
  &=\left.\partial_u^r\log\phi(u)\right|_{u=0}
    =\left.\partial_u^r\psi^{(r)}(u)\right|_{u=0}\\
  &=\sum_{p=1}^r\frac{(-1)^{p-1}}p
    \left.\partial_u^r(\phi(u)-1)^p\right|_{u=0}.
\end{align*}
Write
\begin{align*}
  \left.\partial_u^r(\phi(u)-1)^p\right|_{u=0}
  =R_{p,r}\!\left(E[X_t^\epsilon\mid\mathscr Y_t^{\lambda,\epsilon}],\ldots,
    E[(X_t^\epsilon)^{\otimes r}\mid\mathscr Y_t^{\lambda,\epsilon}]\right),
\end{align*}
where $R_{p,r}$ is the universal tensor-valued polynomial obtained from the repeated product rule, including the factors of $i$ from the derivatives of $\phi$. 

Subtracting the deterministic vector $X_t^{[0]}$ does not change the positive-order coefficients of the cumulant expansion. Thus, by \eqref{eq:expansion}, $\mathsf A_{r,t}^{[k]}~(1\leq k\leq n)$ is the coefficient of the explicit power $\epsilon^k$ in 
\begin{align*}
  \frac1{i^r}\sum_{p=1}^r\frac{(-1)^{p-1}}pR_{p,r}\!\left(C_t(\epsilon)N_t^{(1)}(\epsilon),\ldots,
    C_t(\epsilon)N_t^{(r)}(\epsilon)\right),
\end{align*}
where
\begin{align*}
  N_t^{(j)}(z)=\tilde{\mathsf E}_t\left[
    \left(\mathbb X_t(z)\right)^{\otimes j}
    \mathbb{J}_t(z)\right].
\end{align*}

By the definition of $G_t$,
\begin{align*}
  \left.\frac1{i^j}\partial_u^jG_t(u,z)\right|_{u=0}
  =\begin{cases}
    N_t^{(j)}(z),&1\leq j\leq n,\\
    0,&j>n.
  \end{cases}
\end{align*}
Since $N_t^{(j)}(z)$ is divisible by $z^j$, the normalized derivatives above
and $N_t^{(j)}(z)$ have the same coefficients through degree $n$.

On the other hand, $C_t(z)G_t(0,z)-1$ has no terms through degree $n$ in $z$. Hence, by \eqref{eq:def-Ktilde},
\begin{align*}
  \partial_u^r\tilde{K}_t(u,z)|_{u=0}
\end{align*}
has the same coefficients through degree $n$ in $z$ as
\begin{align*}
  \sum_{p=1}^r\frac{(-1)^{p-1}}pR_{p,r}\!\left(C_t(z)N_t^{(1)}(z),\ldots,
    C_t(z)N_t^{(r)}(z)\right),
\end{align*}
which gives the desired result.

\end{proof}

As a result of Proposition \ref{prop:A-derivative-K}, we can conclude 
\begin{align*}
  \mathsf A_{r,t}^{[k]}=0~~(r>k),
\end{align*}
since the coefficient of $z^k$ in $K_t^{(k)}(u,z)$ is a polynomial in $u$ of degree at most $k$. However, we can prove a stronger vanishing result. In order to prove this result, let us write
\begin{align*}
  M_t(z)=\int_0^t \mathbb{H}_s(z)^\top R(s)^{-1}(d\mathsf{Y}_s-\tilde{H}_s ds).
\end{align*}
Then, the following lemma holds.
\begin{lemma}\label{lemma-J-bb-expression}
  \begin{align*}
    \mathbb{J}_t(z)\equiv\sum_{\ell=0}^n\frac{1}{\ell!}\left( M_t(z)-\frac{1}{2}[M(z)]_t \right)^\ell\pmod{z^{n+1}}.
  \end{align*}
\end{lemma}
\begin{proof}
  Let us write
  \begin{align*}
    \mathcal{E}_t(z)=\sum_{\ell=0}^n\frac{1}{\ell!}\left( M_t(z)-\frac{1}{2}[M(z)]_t \right)^\ell.
  \end{align*}
  Then, It\^o's formula gives
  \begin{align*}
    d\mathcal{E}_t(z)=&\sum_{\ell=0}^{n-1}\frac{1}{\ell!}\left( M_t(z)-\frac{1}{2}[M(z)]_t \right)^\ell \left\{ dM_t(z)-\frac{1}{2}d[M(z)]_t \right\}\\
    &+\frac{1}{2}\sum_{\ell=0}^{n-2}\frac{1}{\ell!}\left( M_t(z)-\frac{1}{2}[M(z)]_t \right)^\ell d[M(z)]_t \\
    =&\sum_{\ell=0}^{n-1}\frac{1}{\ell!}\left( M_t(z)-\frac{1}{2}[M(z)]_t \right)^\ell dM_t(z)\\
    &-\frac{1}{2}\frac{1}{(n-1)!}\left( M_t(z)-\frac{1}{2}[M(z)]_t \right)^{n-1}d[M(z)]_t\\
    \equiv &\mathcal E_t(z)dM_t(z)\pmod{z^{n+1}},
  \end{align*}
  noting that
  \begin{align*}
    M_t(z)\equiv 0\pmod{z},~~[M(z)]_t\equiv 0\pmod{z^2}.
  \end{align*}

  On the other hand, we have
  \begin{align*}
    d \mathbb{J}_t(z)\equiv\mathbb{J}_t(z)dM_t(z)\pmod{z^{n+1}}
  \end{align*}
  with
  \begin{align*}
    \mathbb{J}_0(z)=\mathcal{E}_0(z)=1.
  \end{align*}
  This yields the desired equivalence.
\end{proof}
\begin{align*}
  N_t(z)=&M_t(z)-\frac{1}{2}[M(z)]_t\\
  =&\int_0^t \mathbb{H}_s(z)^\top R(s)^{-1}(d\mathsf{Y}_s-\tilde{H}_s ds)
  -\frac{1}{2}\int_0^t \mathbb{H}_s(z)^\top R(s)^{-1}\mathbb{H}_s(z)\,ds
\end{align*}
\begin{proposition}\label{prop:K-mixed-cumulants}
  We have
  \begin{align*}
    &K_t(u,z)
    \equiv\sum_{\substack{k,m\geq0\\1\leq k+m\leq n}}
      \frac{i^k}{k!m!}\operatorname{cum}_{\tilde{\mathsf E}_t}\left(
      \underbrace{\mathbb X_t(z),\ldots,\mathbb X_t(z)}_{k\text{ times}},
      \underbrace{N_t(z),\ldots,N_t(z)}_{m\text{ times}}
      \right)[u^{\otimes k}]\pmod{z^{n+1}}.
  \end{align*}
\end{proposition}
\begin{proof}
  Let us introduce
  \begin{align*}
    \hat{K}_t(u,v,z)=\sum_{p=1}^n\frac{(-1)^{p-1}}p
    \left\{\tilde{\mathsf E}_t\left[
    \sum_{\ell=0}^n\frac1{\ell!}
    \left(iu\cdot\mathbb{X}_t(z)+iv N_t(z) \right)^\ell
    \right]-1\right\}^p,
  \end{align*}
  for $u \in \mathbb{R}^{d_1}$ and $v \in \mathbb{R}$. Then, for $k,m\geq0$ with $1\leq k+m\leq n$, the Taylor expansion gives
  \begin{align*}
    \partial_v^m\partial_{u}^k\hat{K}_t(0,0,z)&=\left.\partial_v^m\partial_{u}^k\log \tilde{\mathsf E}_t\left[ \exp\left( iu\cdot\mathbb{X}_t(z)+iv N_t(z) \right) \right]\right|_{u=v=0}\\
    &=i^{k+m}\operatorname{cum}_{\tilde{\mathsf E}_t}\left(
      \underbrace{\mathbb X_t(z),\ldots,\mathbb X_t(z)}_{k\text{ times}},
      \underbrace{N_t(z),\ldots,N_t(z)}_{m\text{ times}}
      \right).
  \end{align*}
  Since $\mathbb X_t(z)$ and $N_t(z)$ are divisible by $z$, terms of total degree greater than $n$ in $(u,v)$ are divisible by $z^{n+1}$. Thus we can write
  \begin{align*}
    &\hat{K}_t(u,v,z)
    \equiv\sum_{\substack{k,m\geq0\\1\leq k+m\leq n}}
      \frac{i^{k+m}v^m}{k!m!}\operatorname{cum}_{\tilde{\mathsf E}_t}\left(
      \underbrace{\mathbb X_t(z),\ldots,\mathbb X_t(z)}_{k\text{ times}},
      \underbrace{N_t(z),\ldots,N_t(z)}_{m\text{ times}}
      \right)[u^{\otimes k}]\pmod{z^{n+1}}.
  \end{align*}
  On the other hand, due to Lemma \ref{lemma-J-bb-expression}, we have
  \begin{align*}
    K_t(u,z)
    \equiv& \sum_{p=1}^n\frac{(-1)^{p-1}}p
    \left\{\tilde{\mathsf E}_t\left[
    \sum_{\ell=0}^n\frac1{\ell!}
    \left(iu\cdot\mathbb{X}_t(z)\right)^\ell
    \sum_{\ell=0}^n\frac{1}{\ell!}N_t(z)^\ell
    \right]-1\right\}^p\\
    &\equiv \sum_{p=1}^n\frac{(-1)^{p-1}}p\left\{\tilde{\mathsf E}_t\left[
    \sum_{\ell=0}^n\frac1{\ell!}
    \left(iu\cdot\mathbb{X}_t(z)+N_t(z)\right)^\ell
    \right]-1\right\}^p\\
    &=\hat K_t(u,-i,z)
    \pmod{z^{n+1}}.
  \end{align*}
  Substituting $v=-i$ into the polynomial identity above and using $i^{k+m}(-i)^m=i^k$ proves the claim.
\end{proof}

For $p\geq 1$, define $N_t^{[p]}$ by
\begin{align}
  \begin{split}
    N_t^{[p]}=&\lambda\int_0^t (H_s^{[p+1]})^\top R(s)^{-1}(d\mathsf{Y}_s-\tilde{H}_s ds)
    -\frac{\lambda^2}{2}\sum_{i,j\geq 1,i+j=p}\int_0^t (H_s^{[i+1]})^\top R(s)^{-1} H_s^{[j+1]}ds.
  \end{split}\label{eq:def-Np}
\end{align}
Then, $N_t^{[p]}$ is the coefficient of $z^p$ in $N_t$ for $1\leq p\leq n$. As a corollary to Propositions \ref{prop:A-derivative-K} and \ref{prop:K-mixed-cumulants}, we obtain the following result.
\begin{corollary}\label{cor-A-expression-cumulant}
  For $1\leq r\leq k\leq n$, we have
  \begin{align*}
    \mathsf A_{r,t}^{[k]}
    ={}&\sum_{m=0}^{k-r}\frac{1}{m!}
    \sum_{\substack{p_1,\ldots,p_r\geq1,\ q_1,\ldots,q_m\geq1\\
      \sum\limits_{a=1}^r p_a+\sum\limits_{b=1}^m q_b=k}}
      \frac{1}{p_1!\cdots p_r!}
      \times\operatorname{cum}_{\tilde{\mathsf E}_t}\left(
      X_t^{[p_1]},\ldots,X_t^{[p_r]},
      N_t^{[q_1]},\ldots,N_t^{[q_m]}
      \right).
  \end{align*}
\end{corollary} 

\begin{lemma}\label{lemma-gaussian-cumulant-vanishing}
  Let $\mathcal H_q$ denote the $q$-th Wiener chaos of a centered Gaussian family. For $m\geq2$ and integers $d_1,\ldots,d_m\geq0$, suppose that
  \begin{align*}
    F_j\in\bigoplus_{q=0}^{d_j}\mathcal H_q,\qquad j=1,\ldots,m.
  \end{align*}
  Then, we have
  \begin{align*}
    \sum_{j=1}^m d_j<2(m-1)~~\Longrightarrow~~\operatorname{cum}(F_1,\ldots,F_m)=0.
  \end{align*}
\end{lemma}
\begin{proof}
  By multilinearity and the vanishing of cumulants with a constant argument, it suffices to consider $F_j\in\mathcal H_{q_j}$ with $1\leq q_j\leq d_j$. By Corollary 7.3.1 of \citet{peccati2011wiener}, only connected pairing diagrams contribute to the cumulant. Such a diagram has $m$ vertices and $\frac12\sum_{j=1}^m q_j$ edges, whereas connectedness requires at least $m-1$ edges. Thus no such diagram exists under the stated inequality.
\end{proof}

\begin{theorem}\label{theorem:A-high-rank-vanishing}
  For $r\geq1$ and $1\leq k\leq n$ with $k<2r-2$, we have
  \begin{align*}
    \mathsf A_{r,t}^{[k]}=0
  \end{align*}
\end{theorem}
\begin{proof}
  First note that $\{V_s\}_{0\leq s\leq t}$ is a Gaussian process under $\tilde{\mathsf E}_t$ (see Theorem 3.7 in \citet{KURISAKI2026104997} and Theorem 2.4 in \citet{kurisaki2026recursions}). Let $\mathcal{H}_k$ denote the $k$-th Wiener chaos space with respect to the centered Gaussian process $\{V_s-\tilde{\mathsf E}_t[V_s]\}_{0\leq s\leq t}$ in $L^2(\tilde{\mathsf E}_t)$. 
  
  Then, using discrete approximations of the $dV$ integrals and the conditional $L^2$ convergence along subsequences described below, it follows recursively from \eqref{eq:Xp} and \eqref{eq:Hk} that 
  \begin{align*}
    X_s^{[p]},H_s^{[p]}\in \bigoplus_{k=0}^p \mathcal{H}_k
  \end{align*}
  for any $p\geq 0$ and $0\leq s\leq t$. Furthermore, if we write
  \begin{align*}
    \bar{\mathsf Y}_t=\int_0^t R(s)^{-1}d\mathsf{Y}_s,~~t_i^n=t\frac{i}{n}
  \end{align*}
  then we have
  \begin{align*}
    \sum_{i=1}^{n}(H_{t_{i-1}^n}^{[p]})^\top (\bar{\mathsf Y}_{t_{i}^n}-\bar{\mathsf Y}_{t_{i-1}^n})\to \int_0^t (H_s^{[p]})^\top R(s)^{-1}d\mathsf{Y}_s
  \end{align*}
  in $L^2(\tilde{\mathsf E}_t)$ along a subsequence as $n\to\infty$, almost surely, by applying the argument following Proposition 4.5 in \citet{kurisaki2026ks}. Since the observation increments are fixed under this conditional expectation, the left-hand side belongs to $\bigoplus_{k=0}^p\mathcal H_k$. By closedness, the right-hand side also belongs to this space. Thus, \eqref{eq:def-Np} gives
  \begin{align*}
    N_t^{[p]}\in  \bigoplus_{k=0}^{p+2} \mathcal{H}_k,\qquad p\geq1.
  \end{align*}
  If $r>k$, the claim follows from Proposition \ref{prop:A-derivative-K}. Otherwise, each cumulant in Corollary \ref{cor-A-expression-cumulant} has $r+m\geq2$ arguments with total chaos degree at most
  \begin{align*}
    \sum_{a=1}^r p_a+\sum_{b=1}^m(q_b+2)
    =k+2m<2(r+m-1).
  \end{align*}
  Applying Lemma \ref{lemma-gaussian-cumulant-vanishing} componentwise under $\tilde{\mathsf E}_t$ shows that every such cumulant vanishes.
\end{proof}

\section{Generalized Zakai and Kushner--Stratonovich equations}
\label{section:generalized-filtering-equations}
In this section, we derive the Kushner--Stratonovich equation needed to describe the evolution of the expectation appearing in $K_t(u,z)$. Since the integrand of the expectation depends on the observation process $\mathsf Y^{\lambda,\epsilon}$, we need to extend the usual Kushner--Stratonovich equation to cover this setting. However, we can follow the derivation in the usual case given in \citet{bain_crisan2009filtering}.

We first introduce process versions of the tilted expectations in
\eqref{eq:def-E-tilde}. Throughout this section, fix
$\lambda,\epsilon\in[0,1)$ and $T>0$. Let ${}^o$ denote optional projection
onto the observation filtration $\{\mathscr Y_t^{\lambda,\epsilon}\}_{0\leq t\leq T}$
under $\mathsf Q^{\lambda,\epsilon}$, with the usual augmentation.

For a nonnegative real-valued measurable process
$U=\{U_t\}_{0\leq t\leq T}$, define the unnormalized tilted projection by
\begin{align}
  \label{eq:def-rho-tilde}
  \tilde\rho_t(U)
  :=\lim_{n\to\infty}{}^o\bigl((U\tilde Z^{\lambda,\epsilon})\wedge n\bigr)_t,
  \qquad 0\leq t\leq T.
\end{align}
The optional projections of the bounded processes are taken in monotone
versions, so that the limit defines an extended nonnegative optional
process. For each fixed $t$, conditional monotone convergence gives
\begin{align}
  \label{eq:rho-tilde-conditional-expectation}
  \tilde\rho_t(U)
  &=\lim_{n\to\infty}
    E_{\mathsf Q^{\lambda,\epsilon}}\!\left[
      (U_t\tilde Z_t^{\lambda,\epsilon})\wedge n
      \,\middle|\,\mathscr Y_t^{\lambda,\epsilon}\right]
      =E_{\mathsf Q^{\lambda,\epsilon}}\!\left[
      U_t\tilde Z_t^{\lambda,\epsilon}
      \,\middle|\,\mathscr Y_t^{\lambda,\epsilon}\right]
  \qquad\text{a.s.}
\end{align}
For a signed process $U$, we impose
\begin{align}
  E_{\mathsf Q^{\lambda,\epsilon}}\!\left[|U_t|\tilde Z_t^{\lambda,\epsilon}\right]<\infty,
  \qquad 0\leq t\leq T. \label{eq:tilted-process-integrability}
\end{align}
and set $\tilde\rho(U)=\tilde\rho(U^+)-\tilde\rho(U^-)$. The same definition applies componentwise to vector- and tensor-valued processes satisfying \eqref{eq:tilted-process-integrability}.

The corresponding normalized tilted projection is
\begin{align}
  \label{eq:def-pi-tilde}
  \tilde\pi_t(U)
  :=\frac{\tilde\rho_t(U)}{\tilde\rho_t(1)},
\end{align}
where $1$ denotes the constant process. The denominator is the optional
version of
\[
E_{\mathsf Q^{\lambda,\epsilon}}[\tilde Z_t^{\lambda,\epsilon}\mid\mathscr Y_t^{\lambda,\epsilon}];
\]
for each fixed $t$, it is finite and strictly positive almost surely.
Combining \eqref{eq:rho-tilde-conditional-expectation} with
\eqref{eq:def-E-tilde}, we obtain
\begin{align}
  \label{eq:pi-tilde-E-tilde}
  \tilde\pi_t(U)
  =\frac{E_{\mathsf Q^{\lambda,\epsilon}}[U_t\tilde Z_t^{\lambda,\epsilon}\mid\mathscr Y_t^{\lambda,\epsilon}]}
         {E_{\mathsf Q^{\lambda,\epsilon}}[\tilde Z_t^{\lambda,\epsilon}\mid\mathscr Y_t^{\lambda,\epsilon}]}
  =\tilde {\mathsf E}_t^{\lambda,\epsilon}[U_t]
  \qquad\text{almost surely}
\end{align}
for every fixed $t\in[0,T]$. The dependence of $\tilde\rho$ and
$\tilde\pi$ on $(\lambda,\epsilon)$ is suppressed in the notation.

\begin{theorem}[Zakai equation]\label{theorem:Zakai}
  Let $\{U_t\}_{t\geq 0}$ be a real-valued continuous semimartingale satisfying \eqref{eq:tilted-process-integrability} and
  \begin{align*}
    dU_t=a_tdt+b_t^\top R(t)^{-1}\bigl(d\mathsf Y_t^{\lambda,\epsilon}-\tilde H_t dt\bigr)+c_t^\top dV_t,
  \end{align*}
  where $a$, $b$ and $c$ are progressively measurable $\mathbb{R}$, $\mathbb{R}^{d_2}$ and $\mathbb{R}^{m_1}$-valued processes such that
  \begin{align*}
    E_{\mathsf Q^{\lambda,\epsilon}}\!\left[
    \int_0^T\tilde Z_s^{\lambda,\epsilon}
    \left\{|a_s|+|b_s|^2+|c_s|^2
    +(1+|U_s|^2)|\tilde H_s|^2\right\}ds
    \right]<\infty.
  \end{align*}
  Then $\tilde\rho(U)$ admits a continuous version satisfying, for $0\leq t\leq T$,
  \begin{align*}
    \tilde{\rho}_t(U)
    ={}&\tilde{\rho}_0(U)+\int_0^t \tilde{\rho}_s(a)ds+\int_0^t \left\{ \tilde{\rho}_s(b)+\tilde{\rho}_s(\tilde{H} U) \right\}^\top R(s)^{-1} d\mathsf Y_s^{\lambda,\epsilon}.
  \end{align*}
\end{theorem}
\begin{proof}
  First note that $\tilde\rho_t(U)$ can also be written as
  \begin{align*}
    \tilde\rho_t(U)=E_{\mathsf Q^{\lambda,\epsilon}}\!\left[
      U_t\tilde Z_t^{\lambda,\epsilon}
      \,\middle|\,\mathscr Y_T^{\lambda,\epsilon}\right]
  \end{align*}
  for $t\leq T$ (see Proposition 3.15 in \citet{bain_crisan2009filtering}).

  For $\delta>0$, let us write
  \begin{align*}
    D_t^\delta=\frac{\tilde{Z}_t^{\lambda,\epsilon}}{1+\delta \tilde{Z}_t^{\lambda,\epsilon}}.
  \end{align*}
  Then, due to
  \begin{align*}
    d\tilde{Z}_t^{\lambda,\epsilon}=\tilde{Z}_t^{\lambda,\epsilon}(\tilde H_t)^\top R(t)^{-1}d\mathsf Y_t^{\lambda,\epsilon},
  \end{align*}
  we have
  \begin{align*}
    &U_tD_t^\delta-\frac{U_0}{1+\delta}\\
    ={}&\int_0^t D_s^\delta a_s\,ds
    +\int_0^t D_s^\delta b_s^\top R(s)^{-1}d\mathsf Y_s^{\lambda,\epsilon}
    +\int_0^t D_s^\delta c_s^\top dV_s\\
    &+\int_0^t \frac{\tilde{Z}_sU_s}
    {(1+\delta\tilde{Z}_s^{\lambda,\epsilon})^2}
    (\tilde H_s)^\top R(s)^{-1}d\mathsf Y_s^{\lambda,\epsilon}
    -\int_0^t \frac{\delta U_s(\tilde{Z}_s^{\lambda,\epsilon})^2}
    {(1+\delta\tilde{Z}_s^{\lambda,\epsilon})^3}
    (\tilde H_s)^\top R(s)^{-1}\tilde H_s\,ds\\
    &-\int_0^t \frac{\delta(\tilde{Z}_s^{\lambda,\epsilon})^2}
    {(1+\delta\tilde{Z}_s^{\lambda,\epsilon})^2}
    b_s^\top R(s)^{-1} \tilde H_s\,ds.
  \end{align*}
  Now let us take the conditional expectation and consider the limit $\delta\to0$.
  For the left-hand side, $|D_t^\delta|\leq\tilde Z_t^{\lambda,\epsilon}$ and conditional dominated convergence give
  \begin{align*}
    E_{\mathsf Q^{\lambda,\epsilon}}\!\left[
      U_tD_t^\delta
      \,\middle|\,\mathscr Y_T^{\lambda,\epsilon}\right]
    \to E_{\mathsf Q^{\lambda,\epsilon}}\!\left[
      U_t\tilde Z_t^{\lambda,\epsilon}
      \,\middle|\,\mathscr Y_T^{\lambda,\epsilon}\right]
    \quad\text{a.s.}
  \end{align*}
  For the $ds$ terms,
  \begin{align*}
    &|D_s^\delta a_s-\tilde Z_sa_s|
    \leq\tilde Z_s^{\lambda,\epsilon}|a_s|,\\
    &\left|\frac{\delta U_s(\tilde Z_s^{\lambda,\epsilon})^2}
    {(1+\delta\tilde Z_s^{\lambda,\epsilon})^3}
    (\tilde H_s)^\top R(s)^{-1}\tilde H_s\right|
    \leq C\tilde Z_s^{\lambda,\epsilon}|U_s||\tilde H_s|^2,\\
    &\left|\frac{\delta(\tilde Z_s^{\lambda,\epsilon})^2}
    {(1+\delta\tilde Z_s^{\lambda,\epsilon})^2}
    b_s^\top R(s)^{-1}\tilde H_s\right|
    \leq C\tilde Z_s^{\lambda,\epsilon}|b_s||\tilde H_s|.
  \end{align*}
  These bounds are integrable by the assumption, using
  $2|U_s|\leq1+|U_s|^2$ and
  $2|b_s||\tilde H_s|\leq|b_s|^2+|\tilde H_s|^2$.
  Dominated convergence and conditional Fubini therefore yield
  \begin{align*}
    &E_{\mathsf Q^{\lambda,\epsilon}}\!\left[
    \int_0^t D_s^\delta a_s\,ds
    \,\middle|\,\mathscr Y_T^{\lambda,\epsilon}\right]
    \to\int_0^t\tilde\rho_s(a)\,ds,\\
    &E_{\mathsf Q^{\lambda,\epsilon}}\!\left[
    \int_0^t\frac{\delta U_s(\tilde Z_s^{\lambda,\epsilon})^2}
    {(1+\delta\tilde Z_s^{\lambda,\epsilon})^3}
    (\tilde H_s)^\top R(s)^{-1}\tilde H_s\,ds
    \,\middle|\,\mathscr Y_T^{\lambda,\epsilon}\right]\to0,\\
    &E_{\mathsf Q^{\lambda,\epsilon}}\!\left[
    \int_0^t\frac{\delta(\tilde Z_s^{\lambda,\epsilon})^2}
    {(1+\delta\tilde Z_s^{\lambda,\epsilon})^2}
    b_s^\top R(s)^{-1}\tilde H_s\,ds
    \,\middle|\,\mathscr Y_T^{\lambda,\epsilon}\right]
    \to0
  \end{align*}
  in $L^1(\mathsf Q^{\lambda,\epsilon})$.

  For the $dV_s$ term, due to
  \begin{align*}
    E_{\mathsf Q^{\lambda,\epsilon}}\left[ \int_0^t (D_s^\delta)^2 |c_s|^2 ds \right]\leq \frac{1}{4\delta}E_{\mathsf Q^{\lambda,\epsilon}}\left[ \int_0^t \tilde Z_s^{\lambda,\epsilon} |c_s|^2 ds \right]<\infty,
  \end{align*}
  Lemma 3.21 of \citet{bain_crisan2009filtering} yields
  \begin{align*}
    E_{\mathsf Q^{\lambda,\epsilon}}\!\left[
    \int_0^t D_s^\delta c_s^\top dV_s
    \,\middle|\,\mathscr Y_T^{\lambda,\epsilon}\right]=0.
  \end{align*}

  Finally, for the $d\mathsf Y_s^{\lambda,\epsilon}$ terms, since we have
  \begin{align*}
    &E_{\mathsf Q^{\lambda,\epsilon}}\left[ \int_0^t|D_s^\delta b_s^\top|^2ds \right]\leq \frac{1}{4\delta}E_{\mathsf Q^{\lambda,\epsilon}}\left[ \int_0^t\tilde Z_s^{\lambda,\epsilon}|b_s|^2ds \right]<\infty,\\
    &E_{\mathsf Q^{\lambda,\epsilon}}\left[ \int_0^t\left|\frac{\tilde Z_sU_s}
    {(1+\delta\tilde Z_s^{\lambda,\epsilon})^2}
    \tilde H_s\right|^2 ds \right]\leq \frac{1}{4\delta}E_{\mathsf Q^{\lambda,\epsilon}}\left[ \int_0^t\tilde Z_s^{\lambda,\epsilon}|U_s\tilde H_s|^2ds \right]<\infty,
  \end{align*}
  Lemma 3.21 in \citet{bain_crisan2009filtering} gives
  \begin{align}
    &E_{\mathsf Q^{\lambda,\epsilon}}\!\left[
    \int_0^t D_s^\delta b_s^\top R(s)^{-1}d\mathsf Y_s^{\lambda,\epsilon}
    \,\middle|\,\mathscr Y_T^{\lambda,\epsilon}\right]
    =\int_0^tE_{\mathsf Q^{\lambda,\epsilon}}\!\left[
     D_s^\delta b_s^\top R(s)^{-1}
    \,\middle|\,\mathscr Y_T^{\lambda,\epsilon}\right]d\mathsf Y_s^{\lambda,\epsilon},\label{eq:Zakai1}\\
    \begin{split}
      &E_{\mathsf Q^{\lambda,\epsilon}}\!\left[
    \int_0^t\frac{\tilde Z_sU_s}
    {(1+\delta\tilde Z_s^{\lambda,\epsilon})^2}
    (\tilde H_s)^\top R(s)^{-1}d\mathsf Y_s^{\lambda,\epsilon}
    \,\middle|\,\mathscr Y_T^{\lambda,\epsilon}\right]\\
    =&\int_0^tE_{\mathsf Q^{\lambda,\epsilon}}\!\left[
    \frac{\tilde Z_sU_s}
    {(1+\delta\tilde Z_s^{\lambda,\epsilon})^2}
    (\tilde H_s)^\top R(s)^{-1}
    \,\middle|\,\mathscr Y_T^{\lambda,\epsilon}\right]d\mathsf Y_s^{\lambda,\epsilon}.
    \end{split}\label{eq:Zakai2}
  \end{align}

  Furthermore, we have
  \begin{align*}
    &|D_s^\delta b_s-\tilde Z_sb_s|
    \leq\tilde Z_s^{\lambda,\epsilon}|b_s|,\\
    &\left|\frac{\tilde Z_sU_s}
    {(1+\delta\tilde Z_s^{\lambda,\epsilon})^2}\tilde H_s
    -\tilde Z_sU_s\tilde H_s\right|
    \leq\tilde Z_s^{\lambda,\epsilon}|U_s\tilde H_s|,
  \end{align*}
  and the assumption gives
  \begin{align*}
    &\int_0^t\tilde\rho_s(|b|)^2ds
    \leq\sup_{0\leq v\leq T}\tilde\rho_v(1)
      \int_0^t\tilde\rho_s(|b|^2)ds<\infty,\\
   &\int_0^t\tilde\rho_s(|U\tilde H|)^2ds
    \leq\sup_{0\leq v\leq T}\tilde\rho_v(1)
      \int_0^t\tilde\rho_s(|U\tilde H|^2)ds<\infty
  \end{align*}
  almost surely. Here we take a c\`adl\`ag version of the martingale
  $\tilde\rho_t(1)=E_{\mathsf Q^{\lambda,\epsilon}}[\tilde Z_T^{\lambda,\epsilon}\mid\mathscr Y_t^{\lambda,\epsilon}]$.
  Therefore, Proposition B.41 of \citet{bain_crisan2009filtering} shows that \eqref{eq:Zakai1} and \eqref{eq:Zakai2} converge to
  \begin{align*}
    &\int_0^t\tilde\rho_s(b)^\top R(s)^{-1}d\mathsf Y_s^{\lambda,\epsilon},\quad
    \int_0^t\tilde\rho_s(U\tilde H)^\top R(s)^{-1}d\mathsf Y_s^{\lambda,\epsilon}
  \end{align*}
  in probability, respectively. This completes the proof.
\end{proof}

\begin{theorem}[Generalized Kushner--Stratonovich equation]
\label{theorem:generalized-KS}
  Define the auxiliary innovation by
  \begin{align*}
    \nu_t
    =\mathsf Y_t^{\lambda,\epsilon}
      -\int_0^t\tilde\pi_s(\tilde H)\,ds.
  \end{align*}
  Let $\{U_t\}_{t\geq0}$ be a real-valued continuous semimartingale satisfying \eqref{eq:tilted-process-integrability} and
  \begin{align*}
    dU_t=a_t\,dt+b_t^\top R(t)^{-1}
      \bigl(d\mathsf Y_t^{\lambda,\epsilon}-\tilde H_t\,dt\bigr)
      +c_t^\top dV_t,
  \end{align*}
  where $a$, $b$ and $c$ are progressively measurable $\mathbb{R}$, $\mathbb{R}^{d_2}$ and $\mathbb{R}^{m_1}$-valued processes such that
  \begin{align*}
    E_{\mathsf Q^{\lambda,\epsilon}}\!\left[
    \int_0^T\tilde Z_s^{\lambda,\epsilon}
    \left\{|a_s|+|b_s|^2+|c_s|^2
    +(1+|U_s|^2)|\tilde H_s|^2\right\}ds
    \right]<\infty.
  \end{align*}
  Then $\tilde\pi(U)$ admits a continuous version satisfying, for $0\leq t\leq T$,
  \begin{align}
    d\tilde\pi_t(U)
    ={}&\tilde\pi_t(a)\,dt+\bigl\{\tilde\pi_t(b)
      +\tilde\pi_t(U\tilde H)
      -\tilde\pi_t(U)\tilde\pi_t(\tilde H)\bigr\}^{\top}
      R(t)^{-1}d\nu_t.
    \label{eq:generalized-KS}
  \end{align}
\end{theorem}
\begin{proof}
  In the same way as Lemma 3.29 in \citet{bain_crisan2009filtering}, one can show that
  \begin{align*}
    \tilde{\rho}_t(1)=\exp\left( \int_0^t\tilde{\pi}_s(\tilde{H}_s)^\top R(s)^{-1}d\mathsf Y_s^{\lambda,\epsilon}-\frac{1}{2}\int_0^t \tilde{\pi}_s(\tilde{H}_s)^\top R(s)^{-1} \tilde{\pi}_s(\tilde{H}_s)ds \right),
  \end{align*}
  and therefore
  \begin{align*}
    d\frac{1}{\tilde{\rho}_t(1)}=\frac{1}{\tilde{\rho}_t(1)}\left\{ -\tilde{\pi}_t(\tilde{H}_t)^\top R(t)^{-1}d\mathsf Y_t^{\lambda,\epsilon}+\tilde{\pi}_t(\tilde{H}_t)^\top R(t)^{-1} \tilde{\pi}_t(\tilde{H}_t)dt \right\}.
  \end{align*}
  Combined with Theorem \ref{theorem:Zakai}, It\^o's formula yields
  \begin{align*}
    d\tilde{\pi}_t(U)=&d\tilde{\rho}_t(U)\frac{1}{\tilde{\rho}_t(1)}+\tilde{\rho}_t(U)d\frac{1}{\tilde{\rho}_t(1)}+d\left\langle \tilde{\rho}(U),\frac{1}{\tilde{\rho}(1)} \right\rangle_t\\
    =&\left\{ \tilde{\rho}_t(a)dt+\bigl(\tilde{\rho}_t(b)+\tilde{\rho}_t(\tilde{H} U)\bigr)^\top R(t)^{-1}d\mathsf Y_t^{\lambda,\epsilon} \right\}\frac{1}{\tilde{\rho}_t(1)}\\
    &+\frac{\tilde{\rho}_t(U)}{\tilde{\rho}_t(1)}\left\{ -\tilde{\pi}_t(\tilde{H}_t)^\top R(t)^{-1}d\mathsf Y_t^{\lambda,\epsilon}+\tilde{\pi}_t(\tilde{H}_t)^\top R(t)^{-1} \tilde{\pi}_t(\tilde{H}_t)dt \right\}\\
    &-\frac{1}{\tilde{\rho}_t(1)}\bigl(\tilde{\rho}_t(b)+\tilde{\rho}_t(\tilde{H} U)\bigr)^\top R(t)^{-1}\tilde{\pi}_t(\tilde{H}_t)dt\\
    =&\tilde{\pi}_t(a)dt+\bigl\{\tilde\pi_t(b)
      +\tilde\pi_t(U\tilde H)
      -\tilde\pi_t(U)\tilde\pi_t(\tilde H)\bigr\}^{\top}
      R(t)^{-1}d\nu_t.
  \end{align*}
\end{proof}

\section{Evolution of the coefficients}
\label{section:coefficient-evolution}
In this section, we study the evolution of $K_t(u,z)$ defined in \eqref{eq:def-generating-K} by applying Theorem \ref{theorem:generalized-KS}. The dependence on $u$ and $z$ is polynomial, so applying $\tilde{\pi}_t$ coefficientwise to the polynomial inside the tilted expectation gives the following representation, valid simultaneously for all $u$ and $z$:
\begin{align}
  K_t(u,z)=\sum_{p=1}^n\frac{(-1)^{p-1}}p
    \left\{\tilde{\pi}_t\left(
    \sum_{\ell=0}^n\frac1{\ell!}
    \left(iu\cdot\mathbb{X}(z)\right)^\ell
    \mathbb{J}(z)
    \right)-1\right\}^p.\label{eq:def-K-pi}
\end{align}

In what follows, $\equiv$ denotes congruence modulo $z^{n}$ unless otherwise stated. First, let us state the evolution of $\mathbb{X}$ and $\mathbb{J}(z)$.
\begin{lemma}\label{lemma:derivatives-of-X-J}
  The following identities hold.\\
  (1) 
  \begin{align}
\begin{split}
  d\mathbb X_t(z)\equiv{}&
  \sum_{\substack{q,j\geq0\\1\leq q+j\leq n}}
    \frac{z^q}{q!j!}a_t^{[j,q]}
    [\mathbb X_t(z)^{\otimes j}]dt
    +\sum_{q=1}^n\sum_{j=0}^{n-q}
    \frac{z^q}{q!j!}b_t^{[j,q]}
    [\mathbb X_t(z)^{\otimes j}]dV_t.
\end{split}\label{eq:evolution-bb-X}
\end{align}
  (2) \begin{align*}
    \tilde{H}_t+\mathbb{H}_t(z)\equiv H_t^{[0]}+\frac{\lambda}{z}\sum_{\substack{q,j\geq0\\1\leq q+j\leq n}}\frac{z^q}{q!j!}h_t^{[j,q]}[\mathbb{X}_t(z)^{\otimes j}].
  \end{align*}
  (3)
  \begin{align}
    d \mathbb{J}_t(z)\equiv \mathbb{J}_t(z)\mathbb{H}_{t}(z)^\top R({t})^{-1}(d\mathsf Y_t-\tilde{H}_t dt).\label{eq:dJ-bb}
  \end{align}
\end{lemma}
\begin{proof}
  (1) By \eqref{eq:def-bb-X}, the coefficient of $z^k$ in $d\mathbb{X}_t(z)$ ($1\leq k \leq n$) is
  \begin{align*}
    \frac{1}{k!}dX_t^{[k]},
  \end{align*}
  where $dX_t^{[k]}$ is given by \eqref{eq:X1} for $k=1$ and by \eqref{eq-Xp} for $k\geq2$.

  On the other hand, the coefficient of $z^k$ in the right-hand side of \eqref{eq:evolution-bb-X} is
  \begin{align*}
  &\sum_{q=0}^k\sum_{j=0}^{k-q}
    \sum_{\substack{i_1,\ldots,i_j\geq1\\i_1+\cdots+i_j=k-q}}
    \frac{a_t^{[j,q]}
    [X_t^{[i_1]}\otimes\cdots\otimes X_t^{[i_j]}]}{q!j!i_1!\cdots i_j!}dt\\
  &+\sum_{q=1}^k\sum_{j=0}^{k-q}
    \sum_{\substack{i_1,\ldots,i_j\geq1\\i_1+\cdots+i_j=k-q}}
    \frac{b_t^{[j,q]}
    [X_t^{[i_1]}\otimes\cdots\otimes X_t^{[i_j]}]}{q!j!i_1!\cdots i_j!}dV_t.
\end{align*}
  The number of distinct permutations of $(i_1,\cdots,i_j)\in\Lambda(m,j)$ is
  \begin{align*}
    \frac{\nu(m,(i_1,\cdots,i_j))j!i_1!\cdots i_j!}{m!}.
  \end{align*}
  Using this identity and the symmetry of the derivatives, comparison with \eqref{eq:X1} and \eqref{eq-Xp} gives \eqref{eq:evolution-bb-X}.\\
  (2) By \eqref{eq:def-H-bb} and \eqref{eq:def-H-tilde}, the coefficient of $z^k$ on the left-hand side is
  \begin{align*}
    \tilde{H}_t
  \end{align*} 
  for $k=0$ and $\lambda H_t^{[k+1]}$ for $k\geq 1$. Thus the desired identity can be shown in the same way as (1).\\
  (3) Differentiating \eqref{eq:def-J-bb} yields
  \begin{align*}
    d\mathbb{J}_t(z)=&\Biggl[1+\sum_{r=1}^{n-1}\int_0^t\int_0^{t_r}\cdots\int_0^{t_2}
      \mathbb{H}_{t_1}(z)^\top R({t_1})^{-1}
      \{d\mathsf Y_{t_1}-\tilde{H}_{t_1}\,dt_1\}\\
    &\times\cdots\times
      \mathbb{H}_{t_r}(z)^\top R({t_r})^{-1}
      \{d\mathsf Y_{t_r}-\tilde{H}_{t_r}\,dt_r\}\Biggr]
       \mathbb{H}_{t}(z)^\top R({t})^{-1}(d\mathsf Y_t-\tilde{H}_t dt)\\
      \equiv &\mathbb{J}_t(z)\mathbb{H}_{t}(z)^\top R({t})^{-1}(d\mathsf Y_t-\tilde{H}_t dt).
  \end{align*}
\end{proof}

Let us introduce the notation
\begin{align}
    &\mathbb{F}_t(u,z)=\sum_{\ell=0}^n\frac1{\ell!}\label{eq:def-F-bb}
    \left(iu\cdot\mathbb{X}_t(z)\right)^\ell,\\
    &\mathbb{G}_t(u,z)=\mathbb{F}_t(u,z)\mathbb{J}_t(z),\label{eq:def-G-bb}\\
    &\mathbb{L}_t(u,z)=\sum_{p=0}^{n-1}(-1)^{p}
    \left\{\tilde{\pi}_t\left(
    \mathbb{G}(u,z)
    \right)-1\right\}^p,\label{eq:def-L-bb}
  \end{align}
  and write
  \begin{align*}
    \theta_t^{u,z}(U)=\mathbb{L}_t(u,z)\tilde{\pi}_t(\mathbb{G}(u,z)U)
  \end{align*}
  for a measurable tensor-valued process $U$ satisfying
  \begin{align*}
    E_{\mathsf Q^{\lambda,\epsilon}}\left[ |U_t||\mathbb{G}_t(u,z)|\tilde{Z}_t^{\lambda,\epsilon} \right]<\infty
  \end{align*}
  for every $u \in \mathbb{R}^{d_1}$, $z \in \mathbb{R}$ and $0\leq t\leq T$.

\begin{proposition}\label{prop:evolution-of-K}
  \begin{align*}
    dK_t(u,z)\equiv{}&
    \left\{iu\cdot\theta_t^{u,z}(\mathbb A(z))
    -\frac12u^\top\theta_t^{u,z}
      (\mathbb B(z)\mathbb B(z)^\top)u\right\}dt\\
    &-\frac12R(t)^{-1}\left[\left\{\theta_t^{u,z}(\tilde H+\mathbb H(z))
      -\tilde\pi_t(\tilde H)\right\}^{\otimes 2}\right]dt\\
    &+\left\{\theta_t^{u,z}(\tilde H+\mathbb H(z))
      -\tilde\pi_t(\tilde H)\right\}^{\top}
      R(t)^{-1}d\nu_t,
  \end{align*}
  where
  \begin{align}
    &\mathbb{A}_t(z)=\sum_{\substack{q,j\geq0\\1\leq q+j\leq n}}\frac{z^q}{q!j!}a_t^{[j,q]}[\mathbb{X}_t(z)^{\otimes j}],\label{eq:def-A-bb}\\
    &\mathbb{B}_t(z)=\sum_{q=1}^n\sum_{j=0}^{n-q}\frac{z^q}{q!j!}b_t^{[j,q]}[\mathbb{X}_t(z)^{\otimes j}],\label{eq:def-B-bb}
  \end{align}
  and
  \begin{align*}
    R(t)^{-1}[v^{\otimes 2}]=v^\top R(t)^{-1} v
  \end{align*}
  for any $v \in \mathbb{R}^{d_2}$.
\end{proposition}
\begin{proof}
  First we note that
  \begin{align}
    \mathbb{L}_t(u,z)=\frac{1-(-1)^{n}
    \left\{\tilde{\pi}_t\left(
    \mathbb{G}(u,z)
    \right)-1\right\}^{n}}{\tilde{\pi}_t\left(
    \mathbb{G}(u,z)
    \right)},\label{eq:Lt}
  \end{align}
  whenever $\tilde\pi_t(\mathbb{G}(u,z))\neq0$, and thus
  \begin{align}
    \mathbb{L}_t(u,z)\tilde{\pi}_t\left(
    \mathbb{G}(u,z)
    \right)\equiv 1. \label{eq:identity-L}
  \end{align}
  Furthermore, we can write
  \begin{align}
    \begin{split}
      &\sum_{p=0}^{n-2}
    (p+1)\left\{1-\tilde{\pi}_t\left(
    \mathbb{G}(u,z)
    \right)\right\}^p\\
    =&\frac{1-n\left\{1-\tilde{\pi}_t\left(
    \mathbb{G}(u,z)
    \right)\right\}^{n-1}+(n-1)\left\{1-\tilde{\pi}_t\left(
    \mathbb{G}(u,z)
    \right)\right\}^n}{\tilde{\pi}_t\left(
    \mathbb{G}(u,z)\right)^2}\equiv\mathbb{L}_t(u,z)^2\mod z^{n-1}.
    \end{split}\label{eq:L2}
  \end{align}
  Combined with 
  \begin{align*}
    &d\tilde{\pi}_t(\mathbb{G}(u,z))=0\mod z,\\
    &d[\tilde{\pi}(\mathbb{G}(u,z))]_t=0\mod z^2,
  \end{align*}
  we obtain
  \begin{align}
      dK_t(u,z)=&\sum_{p=0}^{n-1}
    \left\{1-\tilde{\pi}_t\left(
    \mathbb{G}(u,z)
    \right)\right\}^p d\tilde{\pi}_t(\mathbb{G}(u,z))\notag \\
    &-\frac{1}{2}\sum_{p=0}^{n-2}
    (p+1)\left\{1-\tilde{\pi}_t\left(
    \mathbb{G}(u,z)
    \right)\right\}^p d[\tilde{\pi}(\mathbb{G}(u,z))]_t\notag \\
    \equiv &\mathbb{L}_t(u,z)d\tilde{\pi}_t(\mathbb{G}(u,z))-\frac{1}{2}\mathbb{L}_t(u,z)^2d[\tilde{\pi}(\mathbb{G}(u,z))]_t.\label{eq-dK}
  \end{align}

  Now, let us find $d\tilde{\pi}_t(\mathbb{G}(u,z))$ using Theorem \ref{theorem:generalized-KS}. First, Lemma \ref{lemma:derivatives-of-X-J}(1) gives
  \begin{align*}
    d\mathbb{F}_t(u,z)\equiv&\sum_{\ell=0}^{n-1}\frac1{\ell!}\left(iu\cdot\mathbb{X}_t(z)\right)^\ell iu\cdot \left( \mathbb{A}_t(z)dt+\mathbb{B}_t(z)dV_t \right)\\
    &-\frac{1}{2}\sum_{\ell=0}^{n-2}\frac1{\ell!}\left(iu\cdot\mathbb{X}_t(z)\right)^\ell u^\top \mathbb{B}_t(z)\mathbb{B}_t(z)^\top u dt\\    
    \equiv &\mathbb{F}_t(u,z)\left\{ iu\cdot \left( \mathbb{A}_t(z)dt+\mathbb{B}_t(z)dV_t \right)-\frac{1}{2}u^\top \mathbb{B}_t(z)\mathbb{B}_t(z)^\top u\,dt \right\}.
  \end{align*}
  Here, note that we used $\mathbb{A}_t(z),~\mathbb{B}_t(z)\equiv 0 \mod z$. Coupling this equation with Lemma \ref{lemma:derivatives-of-X-J}(3), we obtain
  \begin{align*}
    &d \mathbb{G}_t(u,z)\equiv \mathbb{G}_t(u,z)\\
    &\times \Biggl\{ \left(iu\cdot\mathbb{A}_t(z)-\frac{1}{2}u^\top\mathbb{B}_t(z)\mathbb{B}_t(z)^\top u\right)dt
      +iu^\top\mathbb{B}_t(z)dV_t
      +\mathbb{H}_{t}(z)^\top R(t)^{-1}(d\mathsf Y_t-\tilde{H}_t dt)\Biggr\}.
  \end{align*}
  Now let us apply Theorem \ref{theorem:generalized-KS}. For $U=\mathbb G(u,z)$, the coefficients are
  \begin{align*}
    a_t&\equiv\mathbb G_t(u,z)\left\{iu\cdot\mathbb{A}_t(z)
      -\frac12u^\top\mathbb{B}_t(z)\mathbb{B}_t(z)^\top u\right\},\\
    b_t&\equiv\mathbb G_t(u,z)\mathbb H_t(z),\qquad
    c_t\equiv i\mathbb G_t(u,z)\mathbb{B}_t(z)^\top u.
  \end{align*}
  If we write 
  \begin{align*}
     &\bar{\mathbb{J}}_t(z)=1+\sum_{r=1}^n\int_0^t\int_0^{t_r}\cdots\int_0^{t_2}
      \mathbb{H}_{t_1}(z)^\top R(t_1)^{-1}\sigma(t_1)\,dW_{t_1}\\
     &\qquad\qquad\times\cdots\times\mathbb{H}_{t_r}(z)^\top R(t_r)^{-1}\sigma(t_r)\,dW_{t_r},
  \end{align*}
  then Girsanov's theorem gives
  \begin{align*}
    &E_{\mathsf Q^{\lambda,\epsilon}}\!\left[
    \int_0^T\tilde Z_s^{\lambda,\epsilon}(1+|\mathbb{G}_s(u,z)|^2)
    \left\{|\mathbb{A}_s(z)|+|\mathbb{B}_s(z)|^2+|\mathbb{H}_{s}(z)|^2
    +|\tilde H_s|^2\right\}ds\right]\\
    =&E\!\left[
    \int_0^T(1+|\mathbb{F}_s(u,z)\bar{\mathbb{J}}_s(z)|^2)
    \left\{|\mathbb{A}_s(z)|+|\mathbb{B}_s(z)|^2+|\mathbb{H}_{s}(z)|^2
    +|\tilde H_s|^2\right\}ds\right]\\
    <&\infty,
  \end{align*}
  noting that 
  \begin{align*}
    \int_0^T E[|X_t^{[k]}|^p]dt<\infty
  \end{align*}
  for any $k\geq 0$ and $p\geq 1$. (See also the argument following Proposition 4.5 in \citet{kurisaki2026ks}.) The same moment estimates give the integrability conditions of Theorem \ref{theorem:generalized-KS} coefficientwise. Hence, applying Theorem \ref{theorem:generalized-KS} to the real and imaginary parts of each coefficient of $z^0,\ldots,z^n$ we obtain
  \begin{align*}
    d\tilde{\pi}_t(\mathbb{G}(u,z))\equiv &\tilde{\pi}_t \left( \mathbb{G}(u,z)\left\{ iu\cdot \mathbb{A}(z)-\frac{1}{2}u^\top \mathbb{B}(z)\mathbb{B}(z)^\top u \right\} \right)dt\\
    &+\bigl\{\tilde\pi_t(\mathbb{G}(u,z)\{\tilde H+\mathbb{H}(z)\})
      -\tilde\pi_t( \mathbb{G}(u,z))\tilde\pi_t(\tilde H)\bigr\}^{\top}
      R(t)^{-1}d\nu_t.
  \end{align*}
  Now, substituting this equation into \eqref{eq-dK} and using \eqref{eq:Lt} gives the desired result.
\end{proof}

Next, we observe that the right-hand side of Proposition \ref{prop:evolution-of-K} can be expressed in terms of derivatives of $K_t(u,z)$ with respect to $u$.

\begin{definition}\label{definition:tensor-Bell}
  For integers $0\leq k\leq n$ and $r,\ell\geq0$, we write 
  \begin{align*}
    \mathsf B_{\ell,r,t}^{[k]}=\left.\frac{i^{-(\ell+r)}}{k!}\partial_z^k \partial_u^r\partial_x^\ell\exp\left( K_t(u+x,z)-K_t(u,z)\right)\right|_{x=u=z=0}.
  \end{align*}
\end{definition}

Then, $\mathsf B_{\ell,r,t}^{[k]}$ is a polynomial in the coefficients $\mathsf A$.
\begin{proposition}\label{prop:Bell-recurrence}
  We have
  \begin{align*}
    \mathsf B_{0,r,t}^{[k]}&=\mathbf1_{\{k=r=0\}},\qquad k\geq0
  \end{align*}
  and 
  \begin{align}
    &\mathsf B_{\ell,r,t}^{[k]}[x^{\otimes\ell},u^{\otimes r}]=\ell!r!\sum_{s=1}^{\ell}\frac1{s!}
    \sum_{\substack{
      a_\nu\geq1,\ b_\nu\geq0,\ q_\nu\geq a_\nu+b_\nu\ (1\leq\nu\leq s)\\
      \sum\limits_{\nu=1}^s a_\nu=\ell,\ \sum\limits_{\nu=1}^s b_\nu=r,\ \sum\limits_{\nu=1}^s q_\nu=k}}
    \prod_{\nu=1}^s
    \frac{\mathsf A_{a_\nu+b_\nu,t}^{[q_\nu]}
      [x^{\otimes a_\nu},u^{\otimes b_\nu}]}{a_\nu!b_\nu!}.\label{eq:expression-B}
  \end{align}
  for $\ell\geq1$. In particular, we have
  \begin{align*}
    \mathsf B_{1,r,t}^{[k]}&=\mathsf A_{r+1,t}^{[k]},\qquad k\geq 1.
  \end{align*}
\end{proposition}
\begin{proof}
  The case $\ell=0$ follows by setting $x=0$ in the defining exponential. Let $\ell \geq 1$. Due to Proposition \ref{prop:A-derivative-K}, we have
  \begin{align*}
    &\left.\partial_z^k \partial_u^r\partial_x^\ell\left\{ K_t(u+x,z)-K_t(u,z) \right\}\right|_{x=u=z=0}=\partial_z^k\partial_u^{r+\ell}K_t(0,0)
    =k!i^{r+\ell}\mathsf A_{r+\ell,t}^{[k]}
  \end{align*}
  for $1\leq k\leq n$, $\ell\geq1$, and $r\geq0$, with the slots grouped according to $x$ and $u$. Therefore,
  \begin{align*}
    &K_t(u+x,z)-K_t(u,z)\equiv\sum_{k=1}^n\sum_{\substack{\ell\geq1,\ r\geq0\\\ell+r\leq k}}
    \frac{i^{\ell+r}}{\ell!r!}\mathsf A_{\ell+r,t}^{[k]}
    [x^{\otimes\ell},u^{\otimes r}]z^k\pmod{z^{n+1}}.
  \end{align*}
  Substituting this expression into the Taylor expansion of $\exp$, only powers $s\leq\ell$ contribute to degree $\ell$ in $x$. Comparing the terms of degrees $(\ell,r,k)$ in $(x,u,z)$ and multiplying by $\ell!r!i^{-(\ell+r)}$ gives the formula. 
\end{proof}

\begin{remark}\label{remark:Bell-vanishing}
  By Theorem \ref{theorem:A-high-rank-vanishing}, the sum in Proposition \ref{prop:Bell-recurrence} may be restricted further to indices satisfying
  \begin{align*}
    q_\nu\geq2(a_\nu+b_\nu)-2,\qquad \nu=1,\ldots,s,
  \end{align*}
  since every excluded product contains a zero $\mathsf A$ factor. Together with $q_\nu\geq a_\nu+b_\nu$ and $s\leq\ell$, these conditions give
  \begin{align*}
    k\geq\ell+r,\qquad k\geq2(\ell+r-s)\geq2r.
  \end{align*}
  Consequently, for $\ell\geq1$, $r\geq0$, and $0\leq k\leq n$,
  \begin{align*}
    \mathsf B_{\ell,r,t}^{[k]}=0
    \qquad\text{if}\qquad k<\max\{\ell+r,2r\}.
  \end{align*}
\end{remark}

\begin{remark}\label{remark:B-expression}
  For $p$- and $q$-tensors $S$ and $T$, and integers $0\leq a\leq p$ and $0\leq b\leq q$, let us introduce the tensor product with ordering $S\otimes_{(a,b)} T$ by
\begin{align*}
  (S\otimes_{(a,b)} T)_{i_1\cdots i_{a}j_1\cdots j_{b}k_{1}\cdots k_{p-a}\ell_1\cdots \ell_{q-b}}=S_{i_1\cdots i_a k_1\cdots k_{p-a}}T_{j_{1}\cdots {j_{b}}\ell_{1}\cdots \ell_{q-b}}.
\end{align*}
Empty index blocks are omitted. Moreover, for $p_i$-tensors $T_i~(i=1,\ldots,n)$ and integers $0\leq a_i\leq p_i$, let us write
\begin{align*}
  \sideset{}{_{(a_1)}}\bigotimes_{i=1}^1T_i&=T_1,\\
  \sideset{}{_{(a_1,\ldots,a_n)}}\bigotimes_{i=1}^nT_i
  &=\left(\sideset{}{_{(a_1,\ldots,a_{n-1})}}\bigotimes_{i=1}^{n-1}T_i\right)
  \otimes_{(a_1+\cdots+a_{n-1},a_n)}T_n,\qquad n\geq2.
\end{align*}
Then, we can write
\begin{align*}
    &\mathsf B_{\ell,r,t}^{[k]}=\ell!r!\sum_{s=1}^{\ell}\frac1{s!}
    \sum_{\substack{
      a_\nu\geq1,\ b_\nu\geq0,\ q_\nu\geq a_\nu+b_\nu\ (1\leq\nu\leq s)\\
      \sum\limits_{\nu=1}^s a_\nu=\ell,\ \sum\limits_{\nu=1}^s b_\nu=r,\ \sum\limits_{\nu=1}^s q_\nu=k}}
    {\rm Sym}_{\ell,r}\left\{ \sideset{}{_{(a_1,\cdots,a_s)}}\bigotimes_{\nu=1}^s
    \frac{\mathsf A_{a_\nu+b_\nu,t}^{[q_\nu]}}{a_\nu!b_\nu!} \right\},
  \end{align*}
  where
  \begin{align*}
    ({\rm Sym}_{\ell,r}T)_{i_1\cdots i_\ell j_1\cdots j_r}
    =\frac1{\ell!r!}\sum_{\sigma\in\mathfrak S_\ell}\sum_{\tau\in\mathfrak S_r}
    T_{i_{\sigma(1)}\cdots i_{\sigma(\ell)}j_{\tau(1)}\cdots j_{\tau(r)}}
  \end{align*}
  for an $(\ell+r)$-tensor $T$. Here $\mathfrak S_j$ denotes the permutation group on $j$ indices, with $\mathfrak S_0$ consisting of the empty permutation.
\end{remark}

\begin{proposition}\label{prop:R-Bell}
  For every $k\geq1$, $r\geq0$, and $0\leq m\leq n$, we have
  \begin{align*}
    &\left.\frac1{m!}\partial_z^m\partial_u^r
    \theta_t^{u,z}(\mathbb{X}(z)^{\otimes k})\right|_{u=z=0}=i^r\mathsf B_{k,r,t}^{[m]}.
  \end{align*}
\end{proposition}
\begin{proof}
  First, note that
  \begin{align*}
    \partial_u^k \mathbb{F}_t(u,z)
    &=i^k\sum_{\ell=0}^{n-k}\frac{1}{\ell !}(iu\cdot \mathbb{X}_t(z))^\ell \mathbb{X}_t(z)^{\otimes k}\\
    &\equiv i^k\mathbb{F}_t(u,z)\mathbb{X}_t(z)^{\otimes k}
      \mod{z^{n+1}},\qquad 0\leq k\leq n,
  \end{align*}
  since $\mathbb{X}_t(z)^{\otimes k}\equiv 0 \mod z^k$. Thus, we have 
  \begin{align*}
    \partial_u^k \mathbb{G}_t(u,z)\equiv i^k \mathbb{G}_t(u,z)\mathbb{X}_t(z)^{\otimes k} \mod z^{n+1}.
  \end{align*}
  This congruence also holds for $k>n$, since its left-hand side is zero and its right-hand side is divisible by $z^k$.
  Therefore, we have
  \begin{align*}
    \theta_t^{u,z}(\mathbb{X}(z)^{\otimes k})&=\mathbb{L}_t(u,z)\tilde{\pi}_t(\mathbb{G}(u,z)\mathbb{X}(z)^{\otimes k})\\
    &\equiv i^{-k}\mathbb{L}_t(u,z)\partial_x^k\tilde{\pi}_t(\mathbb{G}(u+x,z))|_{x=0}\mod z^{n+1},
  \end{align*}
  and hence
  \begin{align*}
    \partial_u^r\theta_t^{u,z}(\mathbb{X}(z)^{\otimes k})
    \equiv i^{-k}\left.\partial_y^r\partial_x^k\left\{ \mathbb{L}_t(u+y,z)\tilde{\pi}_t(\mathbb{G}(u+x+y,z)) \right\}\right|_{x=y=0}\mod z^{n+1}.
  \end{align*}
  
  Since $\tilde\pi_t(\mathbb G(v,0))=1$, the quotient and the branch of the logarithm with $\log 1=0$ used below are well-defined near $(v,z)=(u,0)$. Congruences for nonpolynomial expressions refer to their Taylor coefficients in $z$ at zero.
  On the other hand, we can write
  \begin{align*}
    \mathbb{L}_t(u+y,z)\tilde{\pi}_t(\mathbb{G}(u+x+y,z))&=\sum_{p=0}^{n-1}
    \left\{1-\tilde{\pi}_t\left(
    \mathbb{G}(u+y,z)
    \right)\right\}^p \tilde{\pi}_t(\mathbb{G}(u+x+y,z))\\
    &=\frac{\tilde{\pi}_t(\mathbb{G}(u+x+y,z))}{\tilde{\pi}_t\left(
    \mathbb{G}(u+y,z)
    \right)}\left[ 1- \left\{1-\tilde{\pi}_t\left(
    \mathbb{G}(u+y,z)
    \right)\right\}^n \right].
  \end{align*}
  Therefore, we have
  \begin{align}
    &\left.\partial_x^k\left\{ \mathbb{L}_t(u+y,z)\tilde{\pi}_t(\mathbb{G}(u+x+y,z)) \right\}\right|_{x=0}
    \equiv
    \left.\partial_x^k\frac{\tilde{\pi}_t(\mathbb{G}(u+x+y,z))}{\tilde{\pi}_t\left(
    \mathbb{G}(u+y,z)
    \right)}\right|_{x=0}\pmod{z^{n+1}}\label{eq:dx-L}
  \end{align}
  for $k\geq 1$, noting that
  \begin{align*}
    \partial_u^k \mathbb{G}(u,z)\equiv 0 \mod z^k,\quad 1-\tilde{\pi}_t\left(
    \mathbb{G}(u+y,z)
    \right)\equiv 0 \mod z.
  \end{align*}
  Furthermore, Taylor's formula for the logarithm and \eqref{eq:def-K-pi} give
  \begin{align*}
    \log\tilde\pi_t(\mathbb G(v,z))=K_t(v,z)+z^{n+1}Q_t(v,z),
  \end{align*}
  where $Q_t$ is smooth near $(u,0)$. Hence the right-hand side of \eqref{eq:dx-L} equals
  \begin{align*}
    &\left.\partial_x^k\exp\left( \log\tilde\pi_t(\mathbb{G}(u+x+y,z))-\log\tilde\pi_t(\mathbb{G}(u+y,z))\right)\right|_{x=0}\\
    \equiv &\left.\partial_x^k\exp\left( K_t(u+x+y,z)-K_t(u+y,z)\right)\right|_{x=0}
    \mod z^{n+1}.
  \end{align*}
  The remainder remains divisible by $z^{n+1}$ after differentiation in $x$ and $y$. Applying $\partial_y^r\partial_z^m/m!$, evaluating at $u=y=z=0$, and using Definition \ref{definition:tensor-Bell} completes the proof, since $m\leq n$ and $i^{-k}i^{k+r}=i^r$.
\end{proof}

\begin{definition}\label{definition:tensor-contraction}
  For $a=1,\cdots,m$, let $I_a=(i_{a,1},\ldots,i_{a,p_a})$ and $J_a=(j_{a,1},\ldots,j_{a,q_a})$ be multi-indices with compatible coordinate ranges. For tensors $T=(T_{I_1,\ldots,I_m})$ and $S^a=(S^a_{I_a,J_a})$, define
  \begin{align*}
    T[S^1,\ldots,S^m]&:=T[S^1\otimes\cdots\otimes S^m],\\
    \bigl(T[S^1,\ldots,S^m]\bigr)_{J_1,\ldots,J_m}
    &:=\sum_{I_1,\ldots,I_m}T_{I_1,\ldots,I_m}
      S^1_{I_1,J_1}\cdots S^m_{I_m,J_m}.
  \end{align*}
  Empty index blocks are allowed. The same notation applies to tensor-valued $T$, with its output indices retained.
\end{definition}

\begin{proposition}\label{prop:tensor-contraction}
  For tensors with compatible index blocks,
  \begin{align*}
    T[S^1[U_1],\ldots,S^m[U_m]]
      &=\bigl(T[S^1,\ldots,S^m]\bigr)[U_1,\ldots,U_m].
  \end{align*}
  In particular, for vector-valued $S,T$, matrix-valued $M$, and a bilinear form $R$, with $r$ inputs for $S,M$ and $s$ inputs for $T$,
  \begin{align*}
    u\cdot S[u^{\otimes r}]&=S[u^{\otimes(r+1)}],\\
    u^\top M[u^{\otimes r}]u&=M[u^{\otimes(r+2)}],\\
    R[S[u^{\otimes r}],T[u^{\otimes s}]]
      &=(R[S,T])[u^{\otimes(r+s)}].
  \end{align*}
\end{proposition}
\begin{proof}
  The identities follow by interchanging the finite sums over contracted indices.
\end{proof}

Combining Propositions \ref{prop:evolution-of-K} and \ref{prop:R-Bell}, we can describe the evolution of $\mathsf{A}_{r,t}^{[k]}$ introduced in Section \ref{section:conditional-cumulant}. 
\begin{theorem}\label{theorem:coefficient-tower}
  For integers $1\leq r<k$, we have
  \begin{align*}
    &d\mathsf A_{r,t}^{[k]}\\
    ={}&r\operatorname{Sym}_r\!\Bigl(\bigl\{a_t^{[1,0]}
      -\lambda^2\mathsf A_{2,t}^{[2]}(h_t^{[1,0]})^\top
        R(t)^{-1}h_t^{[1,0]}\bigr\}
        [\mathsf A_{r,t}^{[k]}]\Bigr)dt\\
    &+r\sum_{\substack{q,\ell\geq0\\2\leq q+\ell\leq k-r+1}}\frac1{q!\ell!}
      \operatorname{Sym}_r\!\Bigl(\bigl\{a_t^{[\ell,q]}
        -\lambda^2\mathsf A_{2,t}^{[2]}(h_t^{[1,0]})^\top R(t)^{-1}h_t^{[\ell,q]}\bigr\}
        [\mathsf B_{\ell,r-1,t}^{[k-q]}]\Bigr)dt\\
    &+\binom r2\sum_{\substack{q\geq2,\ \ell\geq0\\3\leq q+\ell\leq k-r+2}}\frac1{q!\ell!}
      \operatorname{Sym}_r\!\left(\beta_t^{[\ell,q]}
        [\mathsf B_{\ell,r-2,t}^{[k-q]}]\right)dt\\
    &-\frac{\lambda^2}{2}\sum_{p=0}^r\binom rp
      \sum_{\substack{k_1\geq p+2,\ k_2\geq r-p+2\\k_1+k_2=k+2}}
      \sum_{\substack{q_1,\ell\geq0\\q_1+\ell\leq k_1-p}}\quad
      \sum_{\substack{q_2,j\geq0\\q_2+j\leq k_2-r+p}}\frac1{q_1!\ell!q_2!j!}\\
    &\quad\times\operatorname{Sym}_r\!\Bigl(R(t)^{-1}\Bigl[
      h_t^{[\ell,q_1]}[\mathsf B_{\ell,p,t}^{[k_1-q_1]}],
      h_t^{[j,q_2]}[\mathsf B_{j,r-p,t}^{[k_2-q_2]}]
      \Bigr]\Bigr)dt\\
    &+\lambda\bigl(R(t)^{-1}d\nu_t\bigr)\Biggl[
      h_t^{[1,0]}[\mathsf A_{r+1,t}^{[k+1]}]
      +\sum_{\substack{q\geq0,\ \ell\geq1\\2\leq q+\ell\leq k-r+1}}\frac1{q!\ell!}
      h_t^{[\ell,q]}
        [\mathsf B_{\ell,r,t}^{[k+1-q]}]\Biggr].
  \end{align*}
  where $\beta(x,\epsilon)=b(x,\epsilon)b(x,\epsilon)^\top$, and $\operatorname{Sym}_r$ denotes normalized symmetrization:
\begin{align*}
  (\operatorname{Sym}_r T)_{j_1,\ldots,j_r}
    :=\frac1{r!}\sum_{\sigma\in\mathfrak S_r}
      T_{j_{\sigma(1)},\ldots,j_{\sigma(r)}},
\end{align*}
where $\mathfrak S_r$ is the set of permutations of $\{1,\ldots,r\}$.
\end{theorem}
\begin{proof}
  Fix integers $1\leq r<k$ and choose $n\geq k+1$. The coefficients $\mathsf A$ and $\mathsf B$ involved below are independent of this choice, by \eqref{eq:cumulant-coefficient-polynomial} and \eqref{eq:expression-B}. First recall that Proposition \ref{prop:A-derivative-K} gives, for $r\geq1$ and $1\leq k\leq n$,
\begin{align*}
  \left.\partial_z^{k}\partial_u^r K_t(u,z)\right|_{u=z=0}=i^r k! \mathsf{A}_{r,t}^{[k]}.
\end{align*}
Thus we can write
\begin{align*}
  K_t(u,z)\equiv K_t(0,z)
    +\sum_{k=1}^{n-1}\sum_{r=1}^k\frac{i^r}{r!}
      \mathsf A_{r,t}^{[k]}[u^{\otimes r}]z^k.
\end{align*}
On the other hand, Propositions \ref{prop:evolution-of-K} and \ref{prop:R-Bell} yield
\begin{align*}
    &dK_t(u,z)\\
    \equiv{}&i\sum_{k=1}^{n-1}\sum_{r=0}^k
      \sum_{\substack{0\leq q\leq k,\ 0\leq\ell\leq k-q\\q+\ell\geq1}}
      \frac{i^rz^k}{q!\ell!r!}a_t^{[\ell,q]}
      [\mathsf B_{\ell,r,t}^{[k-q]}[u^{\otimes r}]][u]dt\\
    &-\frac12\sum_{k=2}^{n-1}\sum_{r=0}^{k-2}
      \sum_{q=2}^k\sum_{\ell=0}^{k-q}
      \frac{i^rz^k}{q!\ell!r!}\beta_t^{[\ell,q]}
      [\mathsf B_{\ell,r,t}^{[k-q]}[u^{\otimes r}]][u^{\otimes2}]dt\\
    &-\frac{\lambda^2}{2}R(t)^{-1}\Biggl[
      \Biggl\{-h_t^{[1,0]}[\mathsf A_{1,t}^{[1]}]
      -h_t^{[0,1]}
      +\sum_{k=1}^n\sum_{r=0}^k\sum_{q=0}^k\sum_{\ell=0}^{k-q}
      \frac{i^rz^{k-1}}{q!\ell!r!}h_t^{[\ell,q]}
      [\mathsf B_{\ell,r,t}^{[k-q]}[u^{\otimes r}]]\Biggr\}^{\otimes2}\Biggr]dt\\
    &+\lambda\Biggl\{-h_t^{[1,0]}[\mathsf A_{1,t}^{[1]}]
      -h_t^{[0,1]}
      +\sum_{k=1}^n\sum_{r=0}^k\sum_{q=0}^k\sum_{\ell=0}^{k-q}
      \frac{i^rz^{k-1}}{q!\ell!r!}h_t^{[\ell,q]}
      [\mathsf B_{\ell,r,t}^{[k-q]}[u^{\otimes r}]]\Biggr\}^{\top}R(t)^{-1}d\nu_t.
\end{align*}
Here, we used
\begin{align*}
  H_t^{[0]}-\tilde\pi_t(\tilde H)
  =-\lambda\left\{h_t^{[1,0]}[\mathsf A_{1,t}^{[1]}]
      +h_t^{[0,1]}\right\},
\end{align*}
and hence
  \begin{align*}
    &\theta_t^{u,z}(\tilde H+\mathbb H(z))-\tilde\pi_t(\tilde H)
    \equiv\lambda\sum_{k=2}^n\sum_{r=0}^k\sum_{q=0}^k\sum_{\ell=0}^{k-q}
      \frac{i^rz^{k-1}}{q!\ell!r!}h_t^{[\ell,q]}
      [\mathsf B_{\ell,r,t}^{[k-q]}][u^{\otimes r}].
\end{align*}
Further applying Proposition \ref{prop:tensor-contraction} gives
\begin{align*}
  &dK_t(u,z)\\
  \equiv{}&i\sum_{k=1}^{n-1}\sum_{r=0}^k
    \sum_{\substack{0\leq q\leq k,\ 0\leq\ell\leq k-q\\q+\ell\geq1}}
    \frac{i^rz^k}{q!\ell!r!}a_t^{[\ell,q]}
    [\mathsf B_{\ell,r,t}^{[k-q]}][u^{\otimes(r+1)}]dt\\
  &-\frac12\sum_{k=2}^{n-1}\sum_{r=0}^{k-2}\sum_{q=2}^k\sum_{\ell=0}^{k-q}
    \frac{i^rz^k}{q!\ell!r!}\beta_t^{[\ell,q]}
    [\mathsf B_{\ell,r,t}^{[k-q]}][u^{\otimes(r+2)}]dt\\
  &-\frac{\lambda^2}{2}\sum_{k_1,k_2=2}^n\sum_{r=0}^{k_1}\sum_{s=0}^{k_2}
    \sum_{\substack{q_1,\ell\geq0\\q_1+\ell\leq k_1}}\sum_{\substack{q_2,j\geq0\\q_2+j\leq k_2}}
    \frac{i^{r+s}z^{k_1+k_2-2}}{q_1!\ell!q_2!j!r!s!}
     R(t)^{-1}\left[
    h_t^{[\ell,q_1]}[\mathsf B_{\ell,r,t}^{[k_1-q_1]}],
    h_t^{[j,q_2]}[\mathsf B_{j,s,t}^{[k_2-q_2]}]
    \right][u^{\otimes(r+s)}]dt\\
  &+\lambda\sum_{k=2}^n\sum_{r=0}^k\sum_{q=0}^k\sum_{\ell=0}^{k-q}
    \frac{i^rz^{k-1}}{q!\ell!r!}(R(t)^{-1}d\nu_t)
    \left[h_t^{[\ell,q]}
    [\mathsf B_{\ell,r,t}^{[k-q]}]\right][u^{\otimes r}].
\end{align*}

Since $\left.\partial_u^r(T[u^{\otimes r}])\right|_{u=0}
=r!\operatorname{Sym}_r T$, comparison of the coefficients of $z^k$ after applying $i^{-r}\partial_u^r|_{u=0}$ gives, for $1\leq r\leq k\leq n-1$,
\begin{align*}
  &d\mathsf A_{r,t}^{[k]}\\
  ={}&r\sum_{\substack{0\leq q\leq k,\ 0\leq\ell\leq k-q\\q+\ell\geq1}}\frac1{q!\ell!}
    \operatorname{Sym}_r\!\left(a_t^{[\ell,q]}
      [\mathsf B_{\ell,r-1,t}^{[k-q]}]\right)dt\\
  &+\binom r2\sum_{q=2}^k\sum_{\ell=0}^{k-q}\frac1{q!\ell!}
    \operatorname{Sym}_r\!\left(\beta_t^{[\ell,q]}
      [\mathsf B_{\ell,r-2,t}^{[k-q]}]\right)dt\\
  &-\frac{\lambda^2}{2}\sum_{\substack{k_1,k_2\geq2\\k_1+k_2=k+2}}\sum_{p=0}^r
    \sum_{\substack{q_1,\ell\geq0\\q_1+\ell\leq k_1}}\sum_{\substack{q_2,j\geq0\\q_2+j\leq k_2}}
    \frac{1}{q_1!\ell!q_2!j!}\binom rp\\
    &\times \operatorname{Sym}_r\!\left(R(t)^{-1}\left[
    h_t^{[\ell,q_1]}[\mathsf B_{\ell,p,t}^{[k_1-q_1]}],
    h_t^{[j,q_2]}[\mathsf B_{j,r-p,t}^{[k_2-q_2]}]
    \right]\right)dt\\
  &+\lambda\sum_{q=0}^{k+1}\sum_{\ell=0}^{k+1-q}\frac1{q!\ell!}
    (R(t)^{-1}d\nu_t)\Biggl[h_t^{[\ell,q]}
      [\mathsf B_{\ell,r,t}^{[k+1-q]}]\Biggr].
\end{align*}
We can rewrite the third term as
\begin{align*}
  &\sum_{\substack{k_1,k_2\geq2\\k_1+k_2=k+2}}\sum_{p=0}^r
    \sum_{\substack{q_1,\ell\geq0\\q_1+\ell\leq k_1}}\sum_{\substack{q_2,j\geq0\\q_2+j\leq k_2}}
    \frac{1}{q_1!\ell!q_2!j!}\binom rp\\
    &\times \operatorname{Sym}_r\!\left(R(t)^{-1}\left[
    h_t^{[\ell,q_1]}[\mathsf B_{\ell,p,t}^{[k_1-q_1]}],
    h_t^{[j,q_2]}[\mathsf B_{j,r-p,t}^{[k_2-q_2]}]
    \right]\right)\\
  ={}&2r\sum_{q=0}^k\sum_{\ell=0}^{k-q}\frac1{q!\ell!}
    \operatorname{Sym}_r\!\left(R(t)^{-1}\left[
    h_t^{[\ell,q]}[\mathsf B_{\ell,r-1,t}^{[k-q]}],
    h_t^{[1,0]}[\mathsf B_{1,1,t}^{[2]}]\right]\right)\\
  &+\sum_{\substack{k_1,k_2\geq2\\k_1+k_2=k+2}}
    \sum_{\substack{0\leq p\leq r\\(k_1,p)\ne(2,1)\\(k_2,r-p)\ne(2,1)}}
    \sum_{\substack{q_1,\ell\geq0\\q_1+\ell\leq k_1}}\sum_{\substack{q_2,j\geq0\\q_2+j\leq k_2}}
    \frac{1}{q_1!\ell!q_2!j!}\binom rp\\
  &\quad\times\operatorname{Sym}_r\!\left(R(t)^{-1}\left[
    h_t^{[\ell,q_1]}[\mathsf B_{\ell,p,t}^{[k_1-q_1]}],
    h_t^{[j,q_2]}[\mathsf B_{j,r-p,t}^{[k_2-q_2]}]
    \right]\right).
\end{align*}
Here, we have separated $(k_1,p)=(2,1)$ and $(k_2,r-p)=(2,1)$, for which $\mathsf B_{j,1,t}^{[2-q_2]}=0$ unless $(q_2,j)=(0,1)$. The two contributions agree after symmetrization, and they do not overlap since $r<k$. 

By Remark \ref{remark:Bell-vanishing}, the remaining sums may be restricted to $k_1\geq p+2$ and $k_2\geq r-p+2$.
Further recalling $\mathsf B_{1,r,t}^{[k]}=\mathsf A_{r+1,t}^{[k]}$ and separating the terms with $(q,\ell)=(0,1)$ in the drift and innovation coefficient give the desired result. The ranges of the sums are reduced by Remark \ref{remark:Bell-vanishing}.
\end{proof}

Now, we observe that $\mathsf{A}_{r,t}^{[k]}$ can be obtained successively using Theorem \ref{theorem:coefficient-tower}. To this end, define the degree of $\mathsf{A}_{r}^{[k]}$ by
\begin{align*}
  {\rm deg}\mathsf{A}_{r}^{[k]}=k-r+1.
\end{align*}
Due to Theorem \ref{theorem:A-high-rank-vanishing}, the following $m+1$ coefficients can be nonzero at degree $m$ ($m\geq 1$):
\begin{align}
  \mathsf{A}_{1}^{[m]},~\mathsf{A}_{2}^{[m+1]},~\cdots,~\mathsf{A}_{m+1}^{[2m]}.\label{eq:element-Am}
\end{align}
Note that degree \(m\) corresponds to order \(\epsilon^{m-1}\) in the expansion of the conditional cumulant of $\frac{X_t^\epsilon-X_t^{[0]}}{\epsilon}$. Let us denote the set of these processes of degree $m$ by $\mathcal{A}_m$. Below, we will show that they are precisely the variables needed to compute $\mathsf{A}_{1}^{[m]}$, the coefficient of $\epsilon^m$ in the conditional expectation of $X_t^\epsilon$.

Furthermore, define the degree of $\mathsf B_{\ell,r}^{[k]}~(\ell\geq1,\ k\geq\ell+r)$ by the highest degree of the $\mathsf A$ coefficients appearing in the right-hand side of \eqref{eq:expression-B}, including those that vanish. Specifically, we can write
\begin{align*}
  {\rm deg}\mathsf B_{\ell,r}^{[k]}=k-\ell-r+1
\end{align*}
since the term with $s=1$ is $\mathsf A_{\ell+r}^{[k]}$, and the indices in the sum in \eqref{eq:expression-B} satisfy
\begin{align*}
  q_\nu-a_\nu-b_\nu\geq0,\qquad
  \sum_{\nu=1}^s (q_\nu-a_\nu-b_\nu)=k-\ell-r.
\end{align*}
Based on this observation, we reach the following result.
\begin{theorem}\label{theorem:succesive-structure}
  Each coefficient $\mathsf A$ or $\mathsf B$ appearing on the right-hand side of the equation for $\mathsf A_{r,t}^{[k]}$ in Theorem \ref{theorem:coefficient-tower} has degree at most $k-r+1$. Furthermore, each has degree at most $k-r$, except for $\mathsf A_{r,t}^{[k]}$ itself and the new term $\mathsf A_{r+1,t}^{[k+1]}$.

  As a result, for $m\geq2$, each term of $\mathcal{A}_m$ can be obtained in the reverse order of \eqref{eq:element-Am} if $\mathcal{A}_1,\ldots,\mathcal{A}_{m-1}$ are given.
\end{theorem}
\begin{proof}
  Put $m=k-r+1>1$. The degrees of the quantities on the right-hand side of Theorem \ref{theorem:coefficient-tower} are bounded as follows. When two factors are listed, the bound applies to each factor. Factors with first index zero satisfy $\mathsf B_{0,s,t}^{[v]}=\mathbf1_{\{s=v=0\}}$ and hence do not affect the degree bounds.
  \begin{center}
    \renewcommand{\arraystretch}{1.5}
    \setlength{\tabcolsep}{4pt}
    \begin{tabular}{c|c|c|c|c}
      No. & Quantity & Range & Degree & Bound \\ \hline
      (1) & $\mathsf A_{2,t}^{[2]}$ & & $1$ & $1$ \\ \hline
      (2) & $\mathsf A_{r,t}^{[k]}$ & & $k-r+1$ & $m$ \\ \hline
      (3) & $\mathsf B_{\ell,r-1,t}^{[k-q]}$
        & $q\geq0,\ \ell\geq1,\ q+\ell\geq2$ & $k-q-\ell-r+2$ & $m-1$ \\ \hline
      (4) & $\mathsf B_{\ell,r-2,t}^{[k-q]}$
        & $r\geq2,\ q\geq2,\ \ell\geq1$ & $k-q-\ell-r+3$ & $m-1$ \\ \hline
      (5) & $\begin{gathered}\mathsf B_{\ell,p,t}^{[k_1-q_1]}\\\mathsf B_{j,r-p,t}^{[k_2-q_2]}\end{gathered}$
        & $\begin{gathered}k_1\geq p+2,\ k_2\geq r-p+2\\k_1+k_2=k+2\\q_1,q_2\geq0,\ \ell,j\geq1\end{gathered}$
        & $\begin{gathered}k_1-q_1-\ell-p+1\\k_2-q_2-j-r+p+1\end{gathered}$ & $m-1$ \\ \hline
      (6) & $\mathsf A_{r+1,t}^{[k+1]}$ & & $k-r+1$ & $m$ \\ \hline
      (7) & $\mathsf B_{\ell,r,t}^{[k+1-q]}$
        & $q\geq0,\ \ell\geq1,\ q+\ell\geq2$ & $k-q-\ell-r+2$ & $m-1$ \\ \hline
    \end{tabular}
  \end{center}
  In row (5), the degrees of the two $\mathsf B$ factors satisfy
  \begin{align*}
    k_1-q_1-\ell-p+1&\leq k_1-p=m-(k_2-r+p-1)\leq m-1,\\
    k_2-q_2-j-r+p+1&\leq k_2-r+p=m-(k_1-p-1)\leq m-1.
  \end{align*}
  Rows (1), (2), and (6) follow from the definition of the degree of $\mathsf A$, while rows (3), (4), and (7) follow from ${\rm deg}\mathsf B_{\ell,r}^{[k]}=k-\ell-r+1$.
  Finally, Theorem \ref{theorem:A-high-rank-vanishing} gives $\mathsf A_{m+2,t}^{[2m+1]}=0$. Thus the equation for $\mathsf A_{m+1,t}^{[2m]}$ involves no new coefficient of degree $m$, and the remaining equations can be solved successively in decreasing order of the rank.
\end{proof}

For practical computation, intermediate symmetrizations can be deferred until after all contractions have been performed. For $\ell\geq1$ and $r,k\geq0$, define 
  \begin{align*}
    \bar{\mathsf B}_{\ell,r,t}^{[k]}=\ell!r!\sum_{s=1}^{\ell}\frac1{s!}
    \sum_{\substack{
      a_\nu\geq1,\ b_\nu\geq0,\ q_\nu\geq a_\nu+b_\nu\ (1\leq\nu\leq s)\\
      \sum\limits_{\nu=1}^s a_\nu=\ell,\ \sum\limits_{\nu=1}^s b_\nu=r,\ \sum\limits_{\nu=1}^s q_\nu=k}} \sideset{}{_{(a_1,\cdots,a_s)}}\bigotimes_{\nu=1}^s
    \frac{\mathsf A_{a_\nu+b_\nu,t}^{[q_\nu]}}{a_\nu!b_\nu!}.
  \end{align*}
  and for $\ell=0$
  \begin{align*}
    \bar{\mathsf B}_{0,r,t}^{[k]}=\mathbf1_{\{k=r=0\}}.
  \end{align*}
  By Remark \ref{remark:B-expression}, $\mathsf B_{\ell,r,t}^{[k]}$ is the symmetrization of $\bar{\mathsf B}_{\ell,r,t}^{[k]}$ over the first $\ell$ and last $r$ indices separately.

  Thus, for a smooth scalar-valued function $\phi:\mathbb{R}^{d_1}\to \mathbb{R}$, we have
  \begin{align*}
    \partial^\ell \phi(X_t^{[0]})[\mathsf B_{\ell,r,t}^{[k]}]={\rm Sym}_r\left\{ \partial^\ell \phi(X_t^{[0]})[\bar{\mathsf B}_{\ell,r,t}^{[k]}] \right\},
  \end{align*}
  since $\partial^\ell \phi(X_t^{[0]})$ is symmetric. Applying this identity componentwise gives the following corollary.
  \begin{corollary}\label{cor:cumulant-tower2}
    For integers $1\leq r<k$, we have
    \begin{align*}
    &d\mathsf A_{r,t}^{[k]}\\
    ={}&\operatorname{Sym}_r\Biggl(r\bigl\{a_t^{[1,0]}
      -\lambda^2\mathsf A_{2,t}^{[2]}(h_t^{[1,0]})^\top
        R(t)^{-1}h_t^{[1,0]}\bigr\}
        [\mathsf A_{r,t}^{[k]}]dt\\
    &+r\sum_{\substack{q,\ell\geq0\\2\leq q+\ell\leq k-r+1}}\frac1{q!\ell!}
      \bigl\{a_t^{[\ell,q]}
        -\lambda^2\mathsf A_{2,t}^{[2]}(h_t^{[1,0]})^\top R(t)^{-1}h_t^{[\ell,q]}\bigr\}
        [\bar{\mathsf B}_{\ell,r-1,t}^{[k-q]}]dt\\
    &+\binom r2\sum_{\substack{q\geq2,\ \ell\geq0\\3\leq q+\ell\leq k-r+2}}\frac1{q!\ell!}
      \beta_t^{[\ell,q]}
        [\bar{\mathsf B}_{\ell,r-2,t}^{[k-q]}]dt\\
    &-\frac{\lambda^2}{2}\sum_{p=0}^r\binom rp
      \sum_{\substack{k_1\geq p+2,\ k_2\geq r-p+2\\k_1+k_2=k+2}}
      \sum_{\substack{q_1,\ell\geq0\\q_1+\ell\leq k_1-p}}\quad
      \sum_{\substack{q_2,j\geq0\\q_2+j\leq k_2-r+p}}\frac1{q_1!\ell!q_2!j!}\\
    &\quad\times R(t)^{-1}\Bigl[
      h_t^{[\ell,q_1]}[\bar{\mathsf B}_{\ell,p,t}^{[k_1-q_1]}],
      h_t^{[j,q_2]}[\bar{\mathsf B}_{j,r-p,t}^{[k_2-q_2]}]
      \Bigr]dt\\
    &+\lambda\bigl(R(t)^{-1}d\nu_t\bigr)\Biggl[
      h_t^{[1,0]}[\mathsf A_{r+1,t}^{[k+1]}]
      +\sum_{\substack{q\geq0,\ \ell\geq1\\2\leq q+\ell\leq k-r+1}}\frac1{q!\ell!}
      h_t^{[\ell,q]}
        [\bar{\mathsf B}_{\ell,r,t}^{[k+1-q]}]\Biggr]\Biggr).
  \end{align*}
\end{corollary}

In particular, the Kalman--Bucy filter corresponds to the degree-one coefficients $\mathsf A_{1,t}^{[1]}$ and $\mathsf A_{2,t}^{[2]}$.
\begin{corollary}\label{corollary:Kalman-tower}
  The two coefficients in $\mathcal A_1$ satisfy
\begin{align*}
  d\mathsf A_{1,t}^{[1]}
  ={}&\left\{a_t^{[1,0]}\mathsf A_{1,t}^{[1]}+a_t^{[0,1]}\right\}dt
    +\lambda\mathsf A_{2,t}^{[2]}{h_t^{[1,0]}}^\top
      R(t)^{-1}d\nu_t,\\
  d\mathsf A_{2,t}^{[2]}
  ={}&\operatorname{Sym}_2\Bigl(
      2a_t^{[1,0]}[\mathsf A_{2,t}^{[2]}]
      +\frac12\beta_t^{[0,2]}
      -\lambda^2R(t)^{-1}\left[
      h_t^{[1,0]}[\mathsf A_{2,t}^{[2]}],
      h_t^{[1,0]}[\mathsf A_{2,t}^{[2]}]
      \right]\Bigr)dt\\
  ={}&\Bigl\{a_t^{[1,0]}\mathsf A_{2,t}^{[2]}
      +\mathsf A_{2,t}^{[2]}{a_t^{[1,0]}}^\top
      +\frac12\beta_t^{[0,2]}
      -\lambda^2\mathsf A_{2,t}^{[2]}{h_t^{[1,0]}}^\top
      R(t)^{-1}h_t^{[1,0]}\mathsf A_{2,t}^{[2]}\Bigr\}dt.
\end{align*}
\end{corollary}
In particular, in the case of a linear Gaussian model $(\partial_\epsilon a=\partial_\epsilon^2b=\partial_\epsilon h=0)$, our expansion stops at degree 1 and the result coincides with the exact Kalman--Bucy filter.

\begin{corollary}
  For $m\geq1$, the number of scalar components of $X_t^{[0]}$ and $\mathsf A_{r,t}^{[k]}$ needed for the expansion up to degree $m$ is at most
  \begin{align*}
    d_1+\sum_{j=1}^m\sum_{r=1}^{j+1}\binom{d_1+r-1}{r}
    =\binom{d_1+m+2}{m+1}-m-2.
  \end{align*}
\end{corollary}

\subsection*{Expansion of the conditional characteristic function}
Finally, we observe that the expansion of the conditional characteristic function of $(X_t^\epsilon-X_t^{[0]})/\epsilon$ is given by the expansion coefficients $A_{r,t}^{[k],\epsilon}=\mathsf A_{r,t}^{[k],\epsilon,\epsilon}$.

\begin{proposition}\label{prop:characteristic-expansion}
  For any $n\geq1$, $0\leq t\leq T$ and fixed $u\in\mathbb R^{d_1}$, we have
  \begin{align*}
    &E[e^{iu\cdot (X_t^\epsilon-X_t^{[0]})/\epsilon}|\mathcal{Y}_t^\epsilon]\\
    =&\exp\left(
      iA_{1,t}^{[1],\epsilon}[u]
      -\frac{1}{2}A_{2,t}^{[2],\epsilon}[u^{\otimes2}]
      \right)
      \left\{1+
      \sum_{k=1}^n\epsilon^k
      \sum_{r=1}^{3k}
      \frac{i^r}{r!}P_{r,t}^{[k+r],\epsilon}[u^{\otimes r}]
      \right\}+O_P^T(\epsilon^{n+1}),
  \end{align*}
  where, for $k\geq2$ and $r\geq1$,
  \begin{align*}
    P_{r,t}^{[k],\epsilon}
    ={}&r!\sum_{j=1}^{r\wedge \lfloor \frac{k}{2}\rfloor}\frac1{j!}
    \sum_{\substack{
      k_1,\ldots,k_j\geq2,\;k_1+\cdots+k_j=k\\
      1\leq r_\nu\leq\lfloor\frac{k_\nu+2}{2}\rfloor,\;(k_\nu,r_\nu)\ne(2,2)\ (1\leq\nu\leq j)\\
      r_1+\cdots+r_j=r}}
    \operatorname{Sym}_r\left(
      \bigotimes_{\nu=1}^j\frac{A_{r_\nu,t}^{[k_\nu],\epsilon}}{r_\nu!}
      \right).
  \end{align*}
  In particular, expansion coefficients up to $\epsilon^n$ can be constructed with $A_{r,t}^{[k],\epsilon}$ of degree at most $n+1$. 
\end{proposition}
\begin{proof}
  Set
  \begin{align*}
    \bar{X}_t^\epsilon=\frac{X_t^\epsilon-X_t^{[0]}}{\epsilon},\qquad
    R_t^\epsilon=\bar{X}_t^\epsilon-X_t^{[1]},\qquad
    R_t^{\epsilon,(n)}=\sum_{k=2}^{n+1}\frac{\epsilon^{k-1}}{k!}X_t^{[k]},
  \end{align*}
  and
  \begin{align*}
    \psi_t^\epsilon(u)=\exp\left(
      iA_{1,t}^{[1],\epsilon}[u]
      -\frac{1}{2}A_{2,t}^{[2],\epsilon}[u^{\otimes2}]
      \right).
  \end{align*}
  Then, Proposition \ref{prop-expansion-of-SDE} and Theorem 4.10 in \citet{kurisaki2026ks} give
  \begin{align}
    &\psi_t^\epsilon(u)^{-1}E[e^{iu\cdot\bar X_t^\epsilon}|\mathcal Y_t^\epsilon]\nonumber\\
    ={}&E\left[\psi_t^\epsilon(u)^{-1}e^{iu\cdot X_t^{[1]}}
      \sum_{k=0}^n\frac{(iu\cdot R_t^{\epsilon,(n)})^k}{k!}
      \,\middle|\,\mathcal Y_t^\epsilon\right]+O_P^T(\epsilon^{n+1})\nonumber\\
    \begin{split}\label{eq:expansion-characteristic}
    ={}&\Biggl(\sum_{k=0}^{n}\tilde{\mathsf E}_t^{\epsilon,\epsilon}
      \left[\psi_t^\epsilon(u)^{-1}e^{iu\cdot X_t^{[1]}}
      F_t^{[k]}(u)
      \mathbb J_t^{(n+1),\epsilon,\epsilon}(\epsilon)\right]\epsilon^k\Biggr)\\
    &\quad\times\Biggl\{1+\sum_{j=1}^{n}(-1)^j
      \Bigl(\tilde{\mathsf E}_t^{\epsilon,\epsilon}
      [\mathbb J_t^{(n+1),\epsilon,\epsilon}(\epsilon)]-1\Bigr)^j\Biggr\}
      +O_P^T(\epsilon^{n+1}).
    \end{split}
  \end{align}
  where
  \begin{align*}
    &F_t^{[0]}(u)=1,\\
    &F_t^{[k]}(u)=\left.\frac1{k!}\partial_\epsilon^k e^{iu\cdot R_t^\epsilon}\right|_{\epsilon=0}
      =\sum_{m=1}^k\frac{i^m}{m!}
      \sum_{\substack{k_1,\ldots,k_m\geq1\\k_1+\cdots+k_m=k}}
      \prod_{j=1}^m\frac{u\cdot X_t^{[k_j+1]}}{(k_j+1)!}\qquad(k\geq1).
  \end{align*}
  Note also that $\psi_t^\epsilon(u)^{-1}$ is bounded since $A_{2,t}^{[2],\epsilon}$ is deterministic.

  Furthermore, since $X_t^{[1]}$ is Gaussian under $\tilde{\mathsf E}_t^{\epsilon,\epsilon}$ with
  \begin{align*}
    \tilde{\mathsf E}_t^{\epsilon,\epsilon}[X_t^{[1]}]=A_{1,t}^{[1],\epsilon},\qquad
    {\rm Cov}_{\tilde{\mathsf E}_t^{\epsilon,\epsilon}}[X_t^{[1]}]=A_{2,t}^{[2],\epsilon},
  \end{align*}
  we have the Hermite expansion
  \begin{align*}
    \psi_t^\epsilon(u)^{-1}e^{iu\cdot X_t^{[1]}}
    &=\exp\left(iu\cdot(X_t^{[1]}-A_{1,t}^{[1],\epsilon})
      +\frac12A_{2,t}^{[2],\epsilon}[u^{\otimes2}]\right)\\
    &=\sum_{p=0}^{\infty}i^p
      H_p(X_t^{[1]}-A_{1,t}^{[1],\epsilon};A_{2,t}^{[2],\epsilon})[u^{\otimes p}]
  \end{align*}
  in $L^2(\tilde{\mathsf E}_t^{\epsilon,\epsilon})$.
  Here the symmetric Hermite tensor $H_p(x;A_{2,t}^{[2],\epsilon})$ is defined by
  \begin{align*}
    \exp\left(u\cdot x-\frac12A_{2,t}^{[2],\epsilon}[u^{\otimes2}]\right)
    =\sum_{p=0}^{\infty}H_p(x;A_{2,t}^{[2],\epsilon})[u^{\otimes p}].
  \end{align*}
  Therefore, for $0\leq k\leq n$, conditional Cauchy--Schwarz gives
  \begin{align*}
    &\tilde{\mathsf E}_t^{\epsilon,\epsilon}\left[
      \psi_t^\epsilon(u)^{-1}e^{iu\cdot X_t^{[1]}}
      F_t^{[k]}(u)
      \mathbb J_t^{(n+1),\epsilon,\epsilon}(\epsilon)\right]\\
    &\quad=\sum_{p=0}^{\infty}i^p
      \tilde{\mathsf E}_t^{\epsilon,\epsilon}\left[
      H_p(X_t^{[1]}-A_{1,t}^{[1],\epsilon};A_{2,t}^{[2],\epsilon})[u^{\otimes p}]
      F_t^{[k]}(u)
      \mathbb J_t^{(n+1),\epsilon,\epsilon}(\epsilon)\right].
  \end{align*}
  Here, the sum can be restricted to $p\leq3n$ when retaining terms through order $\epsilon^n$. In fact, $F_t^{[k]}(u)$ belongs to the sum of the Wiener chaoses of orders up to $2k$ of $\{V_s- \tilde{\mathsf E}_t^{\epsilon,\epsilon}[V_s]\}_{0\leq s\leq t}$ under $\tilde{\mathsf E}_t^{\epsilon,\epsilon}$, and the terms of $\mathbb J_t^{(n+1),\epsilon,\epsilon}(\epsilon)$ through order $\epsilon^n$ belong to the sum of the Wiener chaoses of orders up to $n$, due to Lemma \ref{lemma-J-bb-expression} and the argument in the proof of Theorem \ref{theorem:A-high-rank-vanishing}. Thus, the orthogonality of Hermite polynomials shows that the coefficients up to $\epsilon^n$ in the summand are $0$ when $p>3n$.

  Thus, H\"older's inequality gives
  \begin{align*}
    &\sum_{p=0}^{\infty}\sum_{k=0}^{n}i^p
      \tilde{\mathsf E}_t^{\epsilon,\epsilon}\left[
      H_p(X_t^{[1]}-A_{1,t}^{[1],\epsilon};A_{2,t}^{[2],\epsilon})[u^{\otimes p}]
      F_t^{[k]}(u)
      \mathbb J_t^{(n+1),\epsilon,\epsilon}(\epsilon)\right]\epsilon^k\\
      =&\sum_{p=0}^{3n}\sum_{k=0}^{n}i^p\tilde{\mathsf E}_t^{\epsilon,\epsilon}\left[
      H_p(X_t^{[1]}-A_{1,t}^{[1],\epsilon};A_{2,t}^{[2],\epsilon})[u^{\otimes p}]
      F_t^{[k]}(u)
      \mathbb J_t^{(n+1),\epsilon,\epsilon}(\epsilon)\right]\epsilon^k+O_P^T(\epsilon^{n+1})\\
      =&\sum_{r=0}^{4n}\frac1{r!}\tilde{\mathsf E}_t^{\epsilon,\epsilon}\left[
       \left.\partial_u^r\left(
        \psi_t^\epsilon(u)^{-1}e^{iu\cdot X_t^{[1]}}e^{iu\cdot R_t^\epsilon}
      \right)\right|_{u=0}
      \mathbb J_t^{(n+1),\epsilon,\epsilon}(\epsilon)\right][u^{\otimes r}]+O_P^T(\epsilon^{n+1})\\
      =&\sum_{r=0}^{4n}\frac1{r!}\tilde{\mathsf E}_t^{\epsilon,\epsilon}\left[
       \left.\partial_u^r\left(
        \psi_t^\epsilon(u)^{-1}e^{iu\cdot \bar{X}_t^\epsilon}
      \right)\right|_{u=0}
      \mathbb J_t^{(n+1),\epsilon,\epsilon}(\epsilon)\right][u^{\otimes r}]+O_P^T(\epsilon^{n+1}),
  \end{align*}
  noting that we have
  \begin{align*}
    &\sup_{\substack{0<\epsilon<1\\0\leq t\leq T}}E\left[ \left|\left.\partial_u^r\left(
        \psi_t^\epsilon(u)^{-1}e^{iu\cdot X_t^{[1]}}
      \right)\right|_{u=0}\right|^p \right]<\infty,\\
    &\sup_{\substack{0<\epsilon<1\\0\leq t\leq T}}\frac{1}{\epsilon^{rp}}E[|\left.\partial_u^r\left(e^{iu\cdot R_t^\epsilon}\right)\right|_{u=0}|^p]<\infty,\\
    &\sup_{\substack{0<\epsilon<1\\0\leq t\leq T}}E[|\mathbb J_t^{(n+1),\epsilon,\epsilon}(\epsilon)|^p]<\infty,\\
    &\tilde{\mathsf E}_t^{\epsilon,\epsilon}\left[
       \left.\partial_u^r\left(
        \psi_t^\epsilon(u)^{-1}e^{iu\cdot X_t^{[1]}}
      \right)\right|_{u=0}
      \otimes\left.\partial_u^q\left(e^{iu\cdot R_t^\epsilon}\right)\right|_{u=0}\mathbb J_t^{(n+1),\epsilon,\epsilon}(\epsilon)\right]=O_P^T(\epsilon^{n+1}),
  \end{align*}
  for any $p\geq1$ and integer $r\geq0$ in the first three estimates, and integers $r>3n$, $0\leq q\leq n$ in the last estimate, and using the argument following Proposition 4.5 in \citet{kurisaki2026ks}.

  This implies that \eqref{eq:expansion-characteristic} agrees with the expansion of
  \begin{align}
    \sum_{r=0}^{4n}\frac1{r!}
      E\left[\left.\partial_u^r\left(
        \psi_t^\epsilon(u)^{-1}e^{iu\cdot \bar{X}_t^\epsilon}
      \right)\right|_{u=0}\,\middle|\,\mathcal Y_t^\epsilon\right]
      [u^{\otimes r}]+O_P^T(\epsilon^{n+1}).\label{eq:finite-characteristic-Taylor}
  \end{align}

  On the other hand, we have the Taylor expansion
  \begin{align*}
    &\log E[e^{iu\cdot\bar X_t^\epsilon}|\mathcal Y_t^\epsilon]
      -\log\psi_t^\epsilon(u)=\sum_{r=1}^{4n} \frac{i^r}{r!}C_{r,t}[u^{\otimes r}]+o(|u|^{4n}),
  \end{align*}
  around $u=0$, where
  \begin{align*}
    &C_{1,t}=E[\bar X_t^\epsilon|\mathcal{Y}_t^\epsilon]-A_{1,t}^{[1],\epsilon},\qquad
    C_{2,t}=\operatorname{Cov}(\bar X_t^\epsilon\mid\mathcal{Y}_t^\epsilon)-A_{2,t}^{[2],\epsilon},\\
    &C_{r,t}={\rm cum}_r(\bar X_t^\epsilon|\mathcal{Y}_t^\epsilon)\qquad (r\geq 3),
  \end{align*}
  and thus
  \begin{align*}
    \psi_t^\epsilon(u)^{-1}E[e^{iu\cdot\bar X_t^\epsilon}|\mathcal Y_t^\epsilon]
    =\sum_{\ell=0}^{4n}\frac1{\ell!}
      \left(\sum_{r=1}^{4n}\frac{i^r}{r!}C_{r,t}[u^{\otimes r}]\right)^\ell+o(|u|^{4n}).
  \end{align*}
  This yields, for $1\leq r\leq4n$, the identity
  \begin{align*}
    E\left[\left.\partial_u^r\left(
        \psi_t^\epsilon(u)^{-1}e^{iu\cdot\bar X_t^\epsilon}
      \right)\right|_{u=0}\,\middle|\,\mathcal Y_t^\epsilon\right]
      &=\left.\partial_u^r\left(E\left[
        \psi_t^\epsilon(u)^{-1}e^{iu\cdot \bar{X}_t^\epsilon}
      \,\middle|\,\mathcal Y_t^\epsilon\right]\right)\right|_{u=0}\\
    &=i^r r!\sum_{\ell=1}^r\frac1{\ell!}
      \sum_{\substack{r_1,\ldots,r_\ell\geq1\\r_1+\cdots+r_\ell=r}}
      \operatorname{Sym}_r\left(\bigotimes_{\nu=1}^\ell\frac{C_{r_\nu,t}}{r_\nu!}\right).
  \end{align*}
  For $r=0$, the expectation is $1$.
  Finally, applying \eqref{eq:expansion-cumulant} through order $n+r$ and dividing by $\epsilon^r$ gives the expansion of $C_{r,t}$ through order $\epsilon^n$.
  By Theorem \ref{theorem:A-high-rank-vanishing}, $A_{r,t}^{[q+r],\epsilon}=0$ for $r>q+2$.
  Substituting these expansions into \eqref{eq:finite-characteristic-Taylor} and collecting powers gives the stated coefficients $P_{r,t}^{[k+r],\epsilon}$ and the range $r\leq3k$, proving the conclusion.
\end{proof}
\begin{remark}
  Deriving an expansion of the conditional density of $X_t^\epsilon$ by Fourier inversion requires remainder estimates that justify integration over $u$, beyond estimates for each fixed $u$. Establishing such estimates is beyond the scope of this paper and is left for future work.
\end{remark}

\begin{example}
  Let us consider the case of $m_1=d_1$, $a(x,\epsilon)=a(x,0),\ b(x,\epsilon)=\epsilon I_{d_1}$, and $h(x,\epsilon)=Hx$, where $H$ is a fixed $d_2\times d_1$ matrix:
  \begin{align*}
    &dX_t^\epsilon=a(X_t^\epsilon,0)dt+\epsilon dV_t,\\
    &dY_t^\epsilon=HX_t^\epsilon dt+\sigma(t)dW_t.
  \end{align*}
  For this model, consider the asymptotic expansion of the conditional cumulant
  \begin{align*}
    {\rm cum}_r(X_t^\epsilon|\mathcal{Y}_t^\epsilon)
      =\sum_{k=0}^n A_{r,t}^{[k],\epsilon}\epsilon^k+O_P^T(\epsilon^{n+1}).
  \end{align*}
  In particular, $A_{1,t}^{[k],\epsilon}$ is the coefficient of $\epsilon^k$ in the asymptotic expansion of the conditional expectation $E[X_t^\epsilon|\mathcal{Y}_t^\epsilon]$.
  Here,
  \begin{align*}
    &dX_t^{[0]}=a_t^{[0,0]}dt,\qquad X_0^{[0]}=x_0,\\
    &d\nu_t^{\epsilon,\epsilon}
      =dY_t^\epsilon-H\bigl(X_t^{[0]}+\epsilon A_{1,t}^{[1],\epsilon}\bigr)dt.
  \end{align*}
  All the coefficients displayed below have zero initial values.

  The equations for degree 1 coefficients $A_{1,t}^{[1],\epsilon}$ and $A_{2,t}^{[2],\epsilon}$ are
  \begin{align*}
    dA_{1,t}^{[1],\epsilon}
      ={}&a_t^{[1,0]}A_{1,t}^{[1],\epsilon}dt
        +\epsilon A_{2,t}^{[2],\epsilon}H^\top R(t)^{-1}d\nu_t^{\epsilon,\epsilon},\\
    dA_{2,t}^{[2],\epsilon}
      ={}&\Bigl\{a_t^{[1,0]}A_{2,t}^{[2],\epsilon}
        +A_{2,t}^{[2],\epsilon}{a_t^{[1,0]}}^\top+I_{d_1}-\epsilon^2A_{2,t}^{[2],\epsilon}H^\top R(t)^{-1}H A_{2,t}^{[2],\epsilon}\Bigr\}dt.
  \end{align*}
  Write
  \begin{align*}
    \Gamma_t^\epsilon
      =a_t^{[1,0]}-\epsilon^2A_{2,t}^{[2],\epsilon}H^\top R(t)^{-1}H.
  \end{align*}
  Then, the equations for degree 2 coefficients $A_{1,t}^{[2],\epsilon}$, $A_{2,t}^{[3],\epsilon}$, and $A_{3,t}^{[4],\epsilon}$ are
  \begin{align*}
    dA_{1,t}^{[2],\epsilon}
      ={}&\Gamma_t^\epsilon[A_{1,t}^{[2],\epsilon}]dt
        +\frac12a_t^{[2,0]}
        \left[A_{2,t}^{[2],\epsilon}+\bigl(A_{1,t}^{[1],\epsilon}\bigr)^{\otimes2}\right]dt+\epsilon\bigl(H^\top R(t)^{-1}d\nu_t^{\epsilon,\epsilon}\bigr)[A_{2,t}^{[3],\epsilon}],\\
    dA_{2,t}^{[3],\epsilon}
      ={}&\operatorname{Sym}_2\Biggl(
        2\Gamma_t^\epsilon[A_{2,t}^{[3],\epsilon}]dt
        +2a_t^{[2,0]}
        \left[A_{1,t}^{[1],\epsilon}\otimes_{(1,1)}A_{2,t}^{[2],\epsilon}\right]dt
        +\epsilon\bigl(H^\top R(t)^{-1}d\nu_t^{\epsilon,\epsilon}\bigr)[A_{3,t}^{[4],\epsilon}]\Biggr),\\
    dA_{3,t}^{[4],\epsilon}
      ={}&\operatorname{Sym}_3\Biggl(
        3\Gamma_t^\epsilon[A_{3,t}^{[4],\epsilon}]
        +3a_t^{[2,0]}
        \left[A_{2,t}^{[2],\epsilon}\otimes_{(1,1)}A_{2,t}^{[2],\epsilon}\right]\Biggr)dt.
  \end{align*}
\end{example}

\section{Special case: expansion of unconditional expectations}
\label{section:unconditional}
In this section we consider the special case where $h=0$. In this case, $X^\epsilon$ and $Y^\epsilon$ are independent and thus the conditional expectation $E[\phi(X_t^\epsilon)|\mathcal{Y}_t^\epsilon]$ equals the unconditional expectation $E[\phi(X_t^\epsilon)]$. Our goal is to expand this expectation for the small-noise model
\begin{align}
  dX_t^\epsilon=a(X_t^\epsilon,\epsilon)dt+b(X_t^\epsilon,\epsilon)dV_t,~X_0^\epsilon=x_0,\label{eq:X-unconditonal}
\end{align}
where $x_0$ is a constant.

Asymptotic expansions of this type have been extensively studied in mathematical finance, particularly for option pricing; see, for example, \citet{takahashi1999contingent} and \citet{uchida2004option}. In such applications, $\phi$ represents a payoff and need not be smooth. For a nonsmooth function $\phi$, Malliavin calculus and the theory of generalized Wiener functionals developed by \citet{watanabe1987analysis} and \citet{yoshida1992statistics} provide a rigorous basis under suitable regularity and nondegeneracy conditions. 

One way to obtain explicit formulas is to expand the characteristic function of the normalized fluctuation around its leading Gaussian term, and apply Fourier inversion to obtain a density expansion \citep{takahashi1999contingent}. In particular, for a smooth function $g:\mathbb{R}^{d_1}\to \mathbb{R}$, the computation scheme of \citet{takahashi2012computation} can be described in terms of the following characteristic-function expansion of $G_T^\epsilon=g(X_T^\epsilon)$ at a fixed terminal time $T$:
\begin{align*}
  &E[e^{iu\cdot (G_T^\epsilon-G_T^{[0]})/\epsilon}]\\
  =&\psi_T(u)E\left[ \exp\left( iu (G_T^{[1]}-\mu_T)+\frac{1}{2}u^2\Gamma_T \right)\sum_{\ell=0}^n\frac{1}{\ell!}\left( iu\sum_{j=1}^n \frac{1}{(j+1)!}G_T^{[j+1]}\epsilon^j \right)^\ell \right]+O(\epsilon^{n+1}),
\end{align*}
where
\begin{align*}
  \partial_\epsilon^k G_T^{\epsilon}|_{\epsilon=0}=G_T^{[k]},~~\mu_T=E[G_T^{[1]}],~~\Gamma_T=\operatorname{Var}(G_T^{[1]}),
\end{align*}
and $\psi_T(u)=E[e^{iuG_T^{[1]}}]$ is the characteristic function of the Gaussian reference $G_T^{[1]}$. The required terminal quantities can be written as
\begin{align*}
  \eta_{(i_1,\cdots,i_k)}(T;u)=E\left[ \exp\left( iu (G_T^{[1]}-\mu_T)+\frac{1}{2}u^2\Gamma_T \right)G_T^{[i_1]}\cdots G_T^{[i_k]} \right]
\end{align*}
for $i_1,\cdots,i_k\geq 1$. Each of these quantities is a polynomial in $u$, as follows from the Hermite expansion used in the proof of Proposition \ref{prop:characteristic-expansion}. The authors derive a finite system of linear ODEs for polynomials in $u$ whose values at $t=T$ coincide with $\eta_{(i_1,\cdots,i_k)}(T;u)$. Consequently, differentiating their ODEs with respect to $u$ and evaluating at $u=0$ gives a finite system of linear ODEs for their polynomial coefficients; see \citet[Theorem~1, Remark~2, and equation~(18)]{takahashi2012computation}. Although their results are established for a scalar-valued function $g$, the same approach extends to multidimensional cases \citep[Remark~2]{takahashi2013extension}.

Similar computations are available in the R package \texttt{yuima}, where the function \texttt{ae} constructs asymptotic expansions of characteristic functions and the corresponding densities \citep{yuima2026manual}. Its implementation differs from the weighted-expectation ODEs above, and it uses the Hermite expansion of the Gaussian exponential factor to express the correction coefficients in terms of mixed moments of $X^{[j]}$. Applying It\^o's formula to their products and taking expectations yields a system of ODEs for these moments, which is generated symbolically and solved numerically \citep{guidotti2026ae}.

On the other hand, our approach constructs the same characteristic-function expansion coefficients through the cumulant coefficients $A_{r,t}^{[k]}$. Due to Corollary \ref{cor:cumulant-tower2} with $h=0$, the evolution of $A_{r,t}^{[k]}$ is given as follows.
\begin{proposition}
    For integers $1\leq r<k$, we have
    \begin{align*}
    d\mathsf A_{r,t}^{[k]}
    ={}\operatorname{Sym}_r\Biggl(&r a_t^{[1,0]}[\mathsf A_{r,t}^{[k]}]dt
    +r\sum_{\substack{q,\ell\geq0\\2\leq q+\ell\leq k-r+1}}\frac1{q!\ell!}
         a_t^{[\ell,q]}[\bar{\mathsf B}_{\ell,r-1,t}^{[k-q]}]dt\\
       &+\binom r2\sum_{\substack{q\geq2,\ \ell\geq0\\3\leq q+\ell\leq k-r+2}}\frac1{q!\ell!}
      \beta_t^{[\ell,q]}
        [\bar{\mathsf B}_{\ell,r-2,t}^{[k-q]}]dt\Biggr).
      \end{align*}
\end{proposition}

Also, note that \eqref{eq:expansion} becomes
\begin{align*}
  E[f(X_t^\epsilon)]=\sum_{k=0}^n E[F_t^{[k]}]\epsilon^k+O(\epsilon^{n+1})
  =\sum_{k=0}^n \frac1{k!}\partial_\epsilon^kE[f(X_t^\epsilon)]|_{\epsilon=0}\epsilon^k+O(\epsilon^{n+1})
\end{align*}
under $h=0$, and thus we have
\begin{align}
  A_{r,t}^{[k]}=\frac1{k!}\partial_{\epsilon}^k{\rm cum}_r(X_t^\epsilon)|_{\epsilon=0}=\frac{i^{-r}}{k!}\partial_{\epsilon}^k\partial_u^r\log E[e^{iu\cdot X_t^\epsilon}]|_{u=\epsilon=0}.\label{eq:unconditinal-cumulant}
\end{align}
Furthermore, if we have 
\begin{align*}
  a(x,\epsilon)=a(x,-\epsilon),~~b(x,\epsilon)=-b(x,-\epsilon),
\end{align*}
then $X_t^{\epsilon}$ and $X_t^{-\epsilon}$ have the same distribution. Hence, \eqref{eq:unconditinal-cumulant} implies $A_{r,t}^{[k]}=0$ when $k$ is odd. In particular, including $X_t^{[0]}$, the number of scalar variables needed to compute the expansion up to degree $m$ is at most
\begin{align}
  d_1+\sum_{j=1}^m\sum_{\substack{1\leq r\leq j+1\\j+r\text{ odd}}}
    \binom{d_1+r-1}{r},\qquad m\geq1.\label{eq:max-number}
\end{align}

\subsection*{Comparison of the numbers of ODE variables}
We compare the dimensions of the ODE systems used by the three approaches above to compute the characteristic-function expansion of $(X_T^\epsilon-X_T^{[0]})/\epsilon$ through order 4, using the model
\begin{align*}
  dX_{t,1}^\epsilon
    &=\left(0.1\sqrt{1+(X_{t,1}^\epsilon)^2}+0.05X_{t,1}^\epsilon\right)dt
      +0.5\epsilon\sqrt{1+(X_{t,1}^\epsilon)^2}\,dV_t,~~X_{0,1}^\epsilon=100\\
  dX_{t,2}^\epsilon&=X_{t,1}^\epsilon dt,~~X_{0,2}^\epsilon=0.
\end{align*}
We first examined the one-dimensional case by considering only $X_1^\epsilon$, and then computed the joint expansion of $(X_1^\epsilon,X_2^\epsilon)$ for the two-dimensional case. The terminal value $X_{T,1}^\epsilon$ and the average $X_{T,2}^\epsilon/T$ enter European and arithmetic Asian option payoffs, respectively. 

Table~\ref{tab:unconditional-dimensions} gives the resulting counts. For \texttt{yuima} version 1.15.35, the main-system dimensions were read from the output of \texttt{ae(..., verbose=TRUE)} after its built-in reduction, and the separate auxiliary ODE dimensions were recorded at the solver calls. 

For the method of \citet{takahashi2012computation}, with the multidimensional extension in \citet[Remark~2]{takahashi2013extension}, we implemented the weighted-moment equations as ODEs for their polynomial coefficients in $u$. Repeated products are identified and identically zero coefficients are removed. These are counts for this implementation, not dimensions reported by those authors or minimal-realization claims.

For our cumulant tower, the counts coincide with the theoretical bound in \eqref{eq:max-number}, with $m=n+1$ for expansion order $n$.

\begin{table}[htbp]
  \centering
  \caption{Numbers of scalar ODE variables for the same normalized characteristic-function expansion. In the \texttt{yuima} column, $a+b$ denotes the main system and the separate auxiliary system. In the Takahashi column, it denotes polynomial coefficients and auxiliary variables. The cumulant counts include the deterministic trajectory and Gaussian covariance.}
  \label{tab:unconditional-dimensions}
  \begin{tabular}{cc|r|r|r}
    State & order & \texttt{yuima} & Takahashi & Cumulant tower \\ \hline
    $X_1$ & 1 & $8+2$ & $5+3$ & 4 \\
    $X_1$ & 2 & $25+2$ & $20+3$ & 6 \\
    $X_1$ & 3 & $59+2$ & $57+3$ & 9 \\
    $X_1$ & 4 & $124+2$ & $136+3$ & 12 \\ \hline
    $(X_1,X_2)$ & 1 & $24+6$ & $14+7$ & 11 \\
    $(X_1,X_2)$ & 2 & $147+6$ & $98+7$ & 19 \\
    $(X_1,X_2)$ & 3 & $589+6$ & $438+7$ & 31 \\
    $(X_1,X_2)$ & 4 & $1888+6$ & $1541+7$ & 46
  \end{tabular}
\end{table}
The \texttt{yuima} auxiliary system evolves the trajectory again together with a fundamental matrix, so the two dimensions should not be interpreted as the dimension of a single simultaneous system. The Takahashi auxiliary variables comprise the trajectory, the nonconstant entries of the fundamental matrix, and the Gaussian covariance: $1+1+1$ in the one-dimensional case and $2+2+3$ in the two-dimensional case. The latter fundamental matrix has one identically zero entry and one identically unit entry. 

After dividing the state variables by 100 for numerical scaling, the coefficients reconstructed from the Takahashi and cumulant systems agree to within $10^{-12}$ in absolute value for every order in the table. The native \texttt{yuima} output, computed with 1000 time steps, agrees with these coefficients to within $5\times10^{-10}$ in the same scaling.

The cumulant tower substantially reduces the number of ODE variables in both examples, with a larger relative reduction at higher expansion orders. At order four in the two-dimensional example, the total number of variables, including auxiliary variables, is reduced by approximately $97.6\%$ relative to the \texttt{yuima} implementation and $97.0\%$ relative to the weighted-moment implementation described above.

\section*{Acknowledgments}
The author was supported by the RIKEN Special Postdoctoral Researcher Program.

\bibliographystyle{plainnat}
\bibliography{ref}
\end{document}

%% file: figures/cumulant_tower.tex
\begin{tikzpicture}[>=Stealth,y=1.05cm,font=\small,
  every node/.style={inner sep=1pt,anchor=base}]
  \node[anchor=base east] at (3.346,0.000) {$E[X_t^\epsilon\mid\mathcal Y_t^\epsilon]\! =\!$};
  \node at (3.800,0.000) {$X_t^{[0]}$};
  \node at (4.450,0.000) {$+$};
  \node (gmean) at (5.150,0.000) {$\epsilon A_1^{[1]}$};
  \node at (5.850,0.000) {$+$};
  \node (mean) at (6.550,0.000) {$\epsilon^2 A_1^{[2]}$};
  \node at (7.950,0.000) {$+\cdots$};

  \node[anchor=base east] at (3.346,-0.900) {$\operatorname{Cov}(X_t^\epsilon\mid\mathcal Y_t^\epsilon)\! =\!$};
  \node (gcov) at (6.550,-0.900) {$\epsilon^2 A_2^{[2]}$};
  \node at (7.250,-0.900) {$+$};
  \node (cov) at (7.950,-0.900) {$\epsilon^3 A_2^{[3]}$};
  \node at (9.350,-0.900) {$+\cdots$};

  \node[anchor=base east] at (3.346,-1.800) {${\rm cum}_3(X_t^\epsilon\mid\mathcal Y_t^\epsilon)\! =\!$};
  \node (gzero) at (7.950,-1.800) {$0$};
  \node at (8.650,-1.800) {$+$};
  \node (third) at (9.350,-1.800) {$\epsilon^4 A_3^{[4]}$};
  \node at (10.750,-1.800) {$+\cdots$};

  \node[anchor=base east] at (3.346,-2.700) {${\rm cum}_4(X_t^\epsilon\mid\mathcal Y_t^\epsilon)\! =\!$};
  \node (fourth) at (10.750,-2.700) {$0$};
  \node at (11.450,-2.700) {$+$};
  \node at (12.150,-2.700) {$\epsilon^6 A_4^{[6]}$};
  \node at (12.935,-2.700) {$+\cdots$};

  \draw[->,thick] (gmean)--(gcov);
  \draw[->,thick] (gcov)--(gzero);
  \draw[->,thick] (mean)--(cov);
  \draw[->,thick] (cov)--(third);
  \draw[->,thick] (third)--(fourth);
  \node[anchor=south east] at (3.346,0.650) {degree};
  \node[anchor=south] at (3.800,0.650) {0};
  \node[anchor=south] at (5.150,0.650) {1};
  \node[anchor=south] at (6.550,0.650) {2};
  \node[anchor=south] at (7.950,0.650) {$\cdots$};
  \draw[dashed,thin] (0.850,0.500)--(13.350,0.500);
\end{tikzpicture}

%% file: ref.bib
@book{Liptser2001,
  author    = {Liptser, Robert S. and Shiryaev, Albert N.},
  title     = {Statistics of Random Processes: I. General Theory},
  publisher = {Springer},
  address   = {Berlin, Heidelberg},
  year      = {2001}
}

@book{bain_crisan2009filtering,
  author    = {Bain, Alan and Crisan, Dan},
  title     = {Fundamentals of Stochastic Filtering},
  volume    = {3},
  publisher = {Springer},
  year      = {2009}
}

@article{kurisaki2026ks,
  author        = {Kurisaki, Masahiro},
  title         = {Asymptotic Expansion of the Kallianpur-Striebel Formula under Perturbations of Linear Models},
  journal       = {arXiv preprint arXiv:2608.05842},
  year          = {2026},
  archivePrefix = {arXiv},
  eprint        = {2608.05842},
  primaryClass  = {math.PR},
  doi           = {10.48550/arXiv.2608.05842},
  url           = {https://arxiv.org/abs/2608.05842}
}

@article{kurisaki2026recursions,
  author        = {Kurisaki, Masahiro},
  title         = {Finite-Dimensional Recursions for Small-Noise Expansions in Nonlinear Filtering},
  journal       = {arXiv preprint arXiv:2609.18229},
  year          = {2026},
  archivePrefix = {arXiv},
  eprint        = {2609.18229},
  primaryClass  = {math.PR},
  doi           = {10.48550/arXiv.2609.18229},
  url           = {https://arxiv.org/abs/2609.18229}
}

@article{kurisaki2026lorenz,
  author        = {Kurisaki, Masahiro and Takeda, Kota},
  title         = {A Second-Order Asymptotic Expansion for Nonlinear Filtering in the {Lorenz--63} System},
  journal       = {arXiv preprint arXiv:2610.12512},
  year          = {2026},
  archivePrefix = {arXiv},
  eprint        = {2610.12512},
  primaryClass  = {nlin.CD},
  doi           = {10.48550/arXiv.2610.12512},
  url           = {https://arxiv.org/abs/2610.12512}
}

@article{KURISAKI2026104997,
  title    = {Pathwise representation of the smoothing distribution in continuous-time linear Gaussian models},
  journal  = {Stochastic Processes and their Applications},
  volume   = {199},
  pages    = {104997},
  year     = {2026},
  issn     = {0304-4149},
  doi      = {https://doi.org/10.1016/j.spa.2026.104997},
  url      = {https://www.sciencedirect.com/science/article/pii/S0304414926001298},
  author   = {Masahiro Kurisaki}
}

@article{kalman1961,
  author  = {Kalman, Rudolf E. and Bucy, Richard S.},
  title   = {New Results in Linear Filtering and Prediction Theory},
  journal = {Journal of Basic Engineering},
  volume  = {83},
  number  = {1},
  pages   = {95--108},
  year    = {1961},
  doi     = {10.1115/1.3658902}
}

@article{stratonovich1960,
  author  = {Stratonovich, Ruslan L.},
  title   = {Conditional Markov Processes},
  journal = {Theory of Probability \& Its Applications},
  volume  = {5},
  number  = {2},
  pages   = {156--178},
  year    = {1960},
  doi     = {10.1137/1105015}
}

@article{kushner1964,
  author  = {Kushner, Harold J.},
  title   = {On the Differential Equations Satisfied by Conditional Probabilities of Markov Processes, with Applications},
  journal = {Journal of the Society for Industrial and Applied Mathematics, Series A: Control},
  volume  = {2},
  number  = {1},
  pages   = {106--119},
  year    = {1964},
  doi     = {10.1137/0302009}
}

@article{Zakai1969,
  author  = {Zakai, Moshe},
  title   = {On the Optimal Filtering of Diffusion Processes},
  journal = {Zeitschrift f{\"u}r Wahrscheinlichkeitstheorie und Verwandte Gebiete},
  volume  = {11},
  pages   = {230--243},
  year    = {1969},
  doi     = {10.1007/BF00536382}
}

@article{HAZEWINKEL1983331,
  author  = {Hazewinkel, M. and Marcus, S. I. and Sussmann, H. J.},
  title   = {Nonexistence of Finite-Dimensional Filters for Conditional Statistics of the Cubic Sensor Problem},
  journal = {Systems \& Control Letters},
  volume  = {3},
  number  = {6},
  pages   = {331--340},
  year    = {1983},
  doi     = {10.1016/0167-6911(83)90074-9}
}

@article{picard1991,
  author  = {Picard, Jean},
  title   = {Efficiency of the Extended Kalman Filter for Nonlinear Systems with Small Noise},
  journal = {SIAM Journal on Applied Mathematics},
  volume  = {51},
  number  = {3},
  pages   = {843--885},
  year    = {1991},
  doi     = {10.1137/0151042}
}

@article{2003OcDyn..53..343E,
  author  = {Evensen, Geir},
  title   = {The Ensemble Kalman Filter: Theoretical Formulation and Practical Implementation},
  journal = {Ocean Dynamics},
  volume  = {53},
  number  = {4},
  pages   = {343--367},
  year    = {2003},
  doi     = {10.1007/s10236-003-0036-9}
}

@article{Germani01021988,
  author  = {Germani, Alfredo and Piccioni, Mauro},
  title   = {Semi-Discretization of Stochastic Partial Differential Equations on \(\mathbb{R}^d\) by a Finite-Element Technique},
  journal = {Stochastics},
  volume  = {23},
  number  = {2},
  pages   = {131--148},
  year    = {1988},
  doi     = {10.1080/17442508808833486}
}

@incollection{lototsky2011chaos,
  author    = {Lototsky, Sergey V.},
  title     = {Chaos Approach to Nonlinear Filtering},
  booktitle = {The Oxford Handbook of Nonlinear Filtering},
  pages     = {231--264},
  publisher = {Oxford University Press},
  year      = {2011}
}

@article{gordon1993,
  author  = {Gordon, Neil J. and Salmond, David J. and Smith, Adrian F. M.},
  title   = {Novel Approach to Nonlinear/Non-Gaussian Bayesian State Estimation},
  journal = {IEE Proceedings F: Radar and Signal Processing},
  volume  = {140},
  number  = {2},
  pages   = {107--113},
  year    = {1993},
  doi     = {10.1049/ip-f-2.1993.0015}
}

@inproceedings{4378823,
  author    = {Crisan, Dan},
  title     = {Particle Filters in a Continuous Time Framework},
  booktitle = {2006 IEEE Nonlinear Statistical Signal Processing Workshop},
  pages     = {73--78},
  year      = {2006},
  doi       = {10.1109/NSSPW.2006.4378823}
}

@article{snyder2008obstacles,
  author  = {Snyder, Chris and Bengtsson, Thomas and Bickel, Peter and Anderson, Jeff},
  title   = {Obstacles to High-Dimensional Particle Filtering},
  journal = {Monthly Weather Review},
  volume  = {136},
  number  = {12},
  pages   = {4629--4640},
  year    = {2008},
  doi     = {10.1175/2008MWR2529.1}
}

@book{peccati2011wiener,
  author    = {Peccati, Giovanni and Taqqu, Murad S.},
  title     = {Wiener Chaos: Moments, Cumulants and Diagrams: A Survey with Computer Implementation},
  series    = {Bocconi \& Springer Series},
  volume    = {1},
  publisher = {Springer},
  address   = {Milan},
  year      = {2011},
  doi       = {10.1007/978-88-470-1679-8}
}

@article{watanabe1987analysis,
  author  = {Watanabe, Shinzo},
  title   = {Analysis of {Wiener} Functionals ({Malliavin} Calculus) and Its Applications to Heat Kernels},
  journal = {The Annals of Probability},
  volume  = {15},
  number  = {1},
  pages   = {1--39},
  year    = {1987},
  doi     = {10.1214/aop/1176992255}
}

@article{yoshida1992statistics,
  author  = {Yoshida, Nakahiro},
  title   = {Asymptotic Expansion for Statistics Related to Small Diffusions},
  journal = {Journal of the Japan Statistical Society},
  volume  = {22},
  number  = {2},
  pages   = {139--159},
  year    = {1992},
  doi     = {10.11329/jjss1970.22.139}
}

@article{takahashi1999contingent,
  author  = {Takahashi, Akihiko},
  title   = {An Asymptotic Expansion Approach to Pricing Financial Contingent Claims},
  journal = {Asia-Pacific Financial Markets},
  volume  = {6},
  number  = {2},
  pages   = {115--151},
  year    = {1999},
  doi     = {10.1023/A:1010080610650}
}

@article{uchida2004option,
  author  = {Uchida, Masayuki and Yoshida, Nakahiro},
  title   = {Asymptotic Expansion for Small Diffusions Applied to Option Pricing},
  journal = {Statistical Inference for Stochastic Processes},
  volume  = {7},
  number  = {3},
  pages   = {189--223},
  year    = {2004},
  doi     = {10.1023/B:SISP.0000049093.20850.11}
}

@article{takahashi2012computation,
  author  = {Takahashi, Akihiko and Takehara, Kohta and Toda, Masashi},
  title   = {A General Computation Scheme for a High-Order Asymptotic Expansion Method},
  journal = {International Journal of Theoretical and Applied Finance},
  volume  = {15},
  number  = {6},
  pages   = {1250044},
  year    = {2012},
  doi     = {10.1142/S0219024912500446}
}

@techreport{takahashi2013extension,
  author      = {Takahashi, Akihiko and Toda, Masashi},
  title       = {Note on an Extension of an Asymptotic Expansion Scheme},
  institution = {CIRJE, Faculty of Economics, The University of Tokyo},
  type        = {CIRJE Discussion Paper},
  number      = {CIRJE-F-860},
  year        = {2013},
  note        = {Revised April 9, 2013},
  url         = {https://www.cirje.e.u-tokyo.ac.jp/research/dp/2012/2012cf860.pdf}
}

@article{kushner1976finite,
  author  = {Kushner, Harold J.},
  title   = {Finite Difference Methods for the Weak Solutions of the {Kolmogorov} Equations for the Density of Both Diffusion and Conditional Diffusion Processes},
  journal = {Journal of Mathematical Analysis and Applications},
  volume  = {53},
  number  = {2},
  pages   = {251--265},
  year    = {1976},
  doi     = {10.1016/0022-247X(76)90109-8}
}

@article{lototsky1997spectral,
  author  = {Lototsky, Sergey and Mikulevicius, Remigijus and Rozovskii, Boris L.},
  title   = {Nonlinear Filtering Revisited: A Spectral Approach},
  journal = {SIAM Journal on Control and Optimization},
  volume  = {35},
  number  = {2},
  pages   = {435--461},
  year    = {1997},
  doi     = {10.1137/S0363012993248918}
}

@article{lototsky2006wiener,
  author  = {Lototsky, Sergey V.},
  title   = {{Wiener} Chaos and Nonlinear Filtering},
  journal = {Applied Mathematics and Optimization},
  volume  = {54},
  number  = {3},
  pages   = {265--291},
  year    = {2006},
  doi     = {10.1007/s00245-006-0871-4}
}

@book{evensenDataAssimilationEnsemble2009,
  author    = {Evensen, Geir},
  title     = {Data Assimilation: The Ensemble Kalman Filter},
  publisher = {Springer},
  address   = {Berlin, Heidelberg},
  edition   = {2},
  year      = {2009},
  doi       = {10.1007/978-3-642-03711-5}
}

@article{dewiljesLongTimeStability2018,
  author  = {de Wiljes, Jana and Reich, Sebastian and Stannat, Wilhelm},
  title   = {Long-Time Stability and Accuracy of the Ensemble {Kalman--Bucy} Filter for Fully Observed Processes and Small Measurement Noise},
  journal = {SIAM Journal on Applied Dynamical Systems},
  volume  = {17},
  number  = {2},
  pages   = {1152--1181},
  year    = {2018},
  doi     = {10.1137/17M1119056}
}

@manual{yuima2026manual,
  author = {{YUIMA Project Team}},
  title  = {{yuima}: The {YUIMA} Project Package for {SDEs}},
  year   = {2026},
  note   = {R package version 1.15.35; documentation for ae, aeCharacteristic, aeDensity, and yuima.ae-class},
  url    = {https://yuimaproject.r-universe.dev/yuima/doc/manual.html}
}

@misc{guidotti2026ae,
  author       = {Guidotti, Emanuele},
  title        = {Asymptotic Expansion Implementation in {YUIMA}},
  year         = {2026},
  howpublished = {R source file \texttt{R/ae.R}, YUIMA Project repository},
  note         = {Accessed October 15, 2026},
  url          = {https://github.com/yuimaproject/yuima/blob/master/R/ae.R}
}
